\documentclass{article}
\usepackage[utf8]{inputenc}

\usepackage[font=small, format=plain, labelfont={bf,it}, textfont=it]{caption}
\usepackage{fullpage}
\usepackage{amssymb, amsmath, tikz, stmaryrd, amsfonts, latexsym, amscd, amsthm, enumerate, enumitem, epstopdf, graphicx, caption, float, multirow, url, epstopdf, xcolor, enumerate, fancyhdr, afterpage, dirtytalk}
\usepackage{booktabs}
\usepackage[all]{xy}
\usepackage[backref]{hyperref}
\usepackage[alphabetic,lite]{amsrefs}
\usepackage[ruled,vlined]{algorithm2e}

\usepackage{soul}
\usepackage{cancel}
\usepackage{tikz}
\usepackage{pifont}
\usepackage{rotating}

\usetikzlibrary{positioning}

\tikzset{cross/.style={path picture={
  \draw
    (path picture bounding box.south east)--(path picture bounding box.north west)
    (path picture bounding box.south west)--(path picture bounding box.north east);}}}
\usepackage{pgfplots}    
\usepackage{float}
\usepackage[nameinlink]{cleveref}
\usepackage{latexsym,amsfonts,amssymb,amsmath}
\usepackage{etaremune}
\newcommand{\Sasha}[1]{{\color{magenta} \sf $\heartsuit$ [#1]}}

\DeclareMathOperator{\supp}{supp}
\newtheorem{theorem}{Theorem}[section]
\newtheorem{lemma}[theorem]{Lemma}

\newtheorem{definition}[theorem]{Definition}

\newtheorem{corollary}[theorem]{Corollary}
\newtheorem{claim}[theorem]{Claim}

\newtheorem*{theorem*}{Theorem}
\newtheorem{remark}[theorem]{Remark}

\theoremstyle{definition}
\newtheorem{example}[theorem]{Example}
\newtheorem{ex}[theorem]{Example}

\title{A machine learning approach to commutative algebra: Distinguishing table vs non-table ideals}
\author{Laia Amor\'os, Oleksandra Gasanova and Laura Jakobsson}

\begin{document}

\maketitle

\begin{abstract}
    We propose a novel approach to distinguish table vs non-table ideals by using different machine learning algorithms. We introduce the reader to table ideals, assuming some knowledge on commutative algebra and describe their main properties. We create a data set containing table and non-table ideals, and we use a feedforward neural network model, a decision tree and a graph neural networks for the classification.
	Our results indicate that there exists an algorithm to distinguish table ideals from no-table ideals, and we prove it along some novel results on table ideals.

\end{abstract}

\section{Introduction}

Machine learning (ML) algorithms have been successfully applied in many different areas during the past decade, leaving applications to pure mathematics mathematics as one of the last unexplored areas from this point of view. We know that mathematics can be used to improve machine learning techniques, but what about the other direction? Can machine learning be used to answer deep mathematical questions or learn new patterns from mathematical objects?

In this paper we try to answer these questions with a very specific problem, namely learning to distinguish table versus non-table ideals. 
Ideals show up in many different context. We are interested in monomial ideals that live in some polynomial ring $k[x_1,\ldots,x_n]$, where $k$ denotes a field.
A \textit{table} ideal is an ideal that come from a table, i.e. from certain arrangement of integers that satisfy some conditions \ref{def-table}.
Table ideals we first introduced in \cite{table} in order to generalise the results in \cite{miro2020weak}.
They are interesting because quotients by table ideals possess the Strong Lefschetz Property (SLP) and have a symmetric Hilbert series.
Given a table, an ideal can be easily computed out of it. However, given a monomial ideal is not easy to tell if this can be derived from a table or not.

We are interested in using supervised machine learning (ML) methods to learn mathematical patterns. 
Since the problems we are interested in fall on the discrete side, we will explore classification methods in ML.
Machine learning has been previously used in some areas of mathematics with very interesting results, the most popular area being mathematical physics.
In \cite{DL_landscape} and \cite{ML_string}, the author proposes a paradigm to study areas from mathematical physics such as Calabi-Yau manifolds and vector bundles or quiver representations for gauge theories using deep learning tools.
In \cite{ML_Lie} the authors show that essential information about classical and exceptional Lie algebras can be machine learned.

Only recently the use of ML to other mathematical areas has been explored. For example, in
\cite{ML_alg} they initiate a study to determine whether artificial intelligence (AI) can learn algebraic structures. They focus on finite groups and finite rings.
In \cite{ML_math} the author reviews different applications of supervised learning on labeled data from fields like geometry, representation theory, combinatorics or number theory.
Machine learning is also leading to interesting results in number theory. For example, in 
\cite{ML_NT} the authors take a ML approach to the Birch and Swinnerton-Dyer conjecture.
In \cite{ML_NF} the authors show that basic ML algorithms can learn to predict certain invariants of algebraic number fields with high success rates.
In \cite{ML_satotate} they explore the Sato-Tate conjecture, and they train a Bayesan classifier that can distinguish between Sato-Tate groups given a small number of their $L$-function factors.
In \cite{ML_arithmetic-curves} the authors show that standard ML algorithms can be trained to predict certain invariants of elliptic curves and genus two arithmetic curves.
An approach to approximate algebraic properties associated to lattice simplices is done in \cite{ML_lattices}.
Finally, also some applications of ML to algebraic geometry have been developed in
\cite{ML_dessins}, where they use it to study \textit{dessins d'enfants}.

In this paper 
we introduce the reader to table ideals and pose the problem of distinguishing table versus non-table ideals as a machine learning problem. We then try different classification methods: feedforward neural networks, decision trees and graph neural networks. 
In order to train the different ML models, we need to generate a large set of table ideals, as well as non-table ideals.
It is important to highlight our ML approach. The results we obtained with the three methods pointed to the existence of an algorithm to completely classify table from non-table ideals, because all the methods had an almost perfect success rate.
Moreover, by generating table and non-table ideas new insights from their inner structure can be learnt.
Our approach can be used in other mathematical problems similar in nature.

This paper is organised as follows.
In Section \ref{sec:tables} we define table ideals and discuss their properties, define so-called generalised tables in a normal form and show the bijection between them and table ideals.
In Section \ref{sec:data} we explain how we can represent monomial ideals in a way that can be used to train our ML models, and we generate our data set of table and non-table ideals.
In section \ref{sec:ML1} we consider a feedforward neural network to classify our monomial ideals. 
In \ref{sec:ML2} we consider a decision tree model, and in \ref{sec:ML3} we use a graph neural network. For this we need to adapt the shape of our monomial ideals and express them as graphs. This method gives us more flexibility as we do not need to limit the number of generators in the ideals, as opposed with the other two methods.
Finally in Section \ref{sec:alg} we give an algorithm to completely classify table and non-table ideals.

\section{Table ideals}
\label{sec:tables}
Table ideals are monomial ideals that come from tables, that is, matrices of integers that satisfy some conditions on the entries.
The story behind this class of ideals is the following. In \cite{guerrieri2019lefschetz}, the author described Gorenstein codimension three algebras associated to the Ap\'ery sets of numerical semigroups.  Later, in \cite{miro2020weak}, the authors explored the Lefschetz properties of such algebras via passing to the initial ideal. This initial ideal had $3$ variables and $1$ linear condition on the parameters involved (for details, see \Cref{miro-roig}). In \cite{table}, table ideals were introduced in order to generalise the results in \cite{miro2020weak} to more variables and more linear conditions (colours). It turns out that quotients by table ideals have a symmetric Hilbert series and possess the Strong Lefschetz Property (SLP), or, equivalently, they are modules over the Lie algebra $\mathfrak{sl_2}$.




Given a table, the ideal coming from it in a non-minimal form is easy to compute. 
However, the inverse problem is unsolved: given a monomial ideal, we do not know whether it is a table ideal or not. 
The question is made harder by the fact that a table ideal can often have the minimal generating set much smaller than the generating set given by the definition.  \Cref{sec:mathbgtable} covers the necessary mathematical background and results on table ideals, and in \Cref{sec:new_results} we present tables of a specific type -- tables in a normal form. In \Cref{sec:simplices} we associate a weighted simplicial complex to a table ideal and see how the properties of table ideals translate to those of the corresponding simplicial complexes. Finally, in \Cref{sec:uniqueness} we prove the bijection between table ideals and generalised tables in a normal form.

\subsection{Mathematical background on table ideals}
\label{sec:mathbgtable}

In this section we introduce a class of ideals which we call \textit{table ideals}.

\begin{definition}\label{def-table}
An $(s,n)$-\emph{table}, where $0\le s<n$, is an $(s+1)\times n$ matrix of non-negative integers as in \Cref{cond} such that:
\begin{enumerate}
    \item[$(i)$] $\alpha_{i,j}=0$ for all $i>j$,
    \item[$(ii)$] $\sum_{i=1}^{s}{\alpha_{i,j}}\le d_j$ for all $j$,
    \item[$(iii)$] $d_k=\sum_{i=1}^{k-1}\alpha_{i,k}+\sum_{j=k+1}^n\alpha_{k,j}+\alpha_{k+1,k+1}$ for $1 \le k \le s$, where we set $\alpha_{s+1,s+1}=0$.
\end{enumerate}

We will sometimes refer to an $(s,n)$-table just as a table.
\end{definition}

\begin{figure}[H]
\centering
\begin{tikzpicture}[scale=1.4]

      \draw [red,thick]     (2,-1) -- (9,-1);
      \draw [red,thick]    (2,-1) -- (2,-2);
      \draw [red,thick]  (2,-1)--(1,0);
      
      \draw[green, thick]  (2,-1)-- (3,-2) -- (9,-2);
      \draw[green, thick]  (3,-2) -- (3,-3);
      \draw[green,thick] (2,-1)--(2,0);

      \draw[cyan, thick]  (3,-1)-- (3,-2) -- (4,-3)--(9,-3);
      \draw[cyan, thick]  (4,-4)-- (4,-3);
      \draw[cyan, thick] (3,-1)--(3,0);

      \draw[yellow,thick]  (4,-1)-- (4,-3) --(5,-4)-- (9,-4);
      \draw[yellow,thick]  (5,-4)-- (5,-5);
      \draw[yellow,thick]  (4,-1)-- (4,0);
      \draw[blue,thick]  (5,-1)--(5,-4) --(6,-5)-- (9,-5);
      \draw[blue,thick]  (5,-1)--(5,-0);

    \foreach \x in {1,...,2}{
    \node [above, thin] at (\x,0) {$d_\x$};
    \node [above, thin] at (\x,-1) {$\alpha_{1,\x}$};
    }
    \node [above, thin] at (1,-2) {$0$};
    \node [above, thin] at (2,-2) {$\alpha_{2,2}$};
    
    \node [above, thin] at (4,-1) {$\alpha_{1,s-1}$};
    \node [above, thin] at (5,-1) {$\alpha_{1,s}$};
    \node [above, thin] at (6,-1) {$\alpha_{1,s+1}$};
    \node [above, thin] at (8,-1) {$\alpha_{1,n-1}$};
    \node [above, thin] at (9,-1) {$\alpha_{1,n}$};
    
    \node [above, thin] at (4,-2) {$\alpha_{2,s-1}$};
    \node [above, thin] at (5,-2) {$\alpha_{2,s}$};
    \node [above, thin] at (6,-2) {$\alpha_{2,s+1}$};
    \node [above, thin] at (8,-2) {$\alpha_{2,n-1}$};
    \node [above, thin] at (9,-2) {$\alpha_{2,n}$};

    \node [above, thin] at (4,0) {$d_{s-1}$};    
    \node [above, thin] at (5,0) {$d_{s}$};  
    \node [above, thin] at (6,0) {$d_{s+1}$};  
    \node [above, thin] at (8,0) {$d_{n-1}$};
    \node [above, thin] at (9,0) {$d_{n}$};

    \node [above, thin] at (1,-4) {$0$};
    \node [above, thin] at (2,-4) {$0$};
    
    \node [above, thin] at (4,-4) {$\alpha_{s-1,s-1}$};
    
    \node [above, thin] at (5,-4) {$\alpha_{s-1,s}$};
    
    \node [above, thin] at (6,-4) {$\alpha_{s-1,s+1}$};
    \node [above, thin] at (8,-4) {$\alpha_{s-1,n-1}$};
    \node [above, thin] at (9,-4) {$\alpha_{s-1,n}$};

    \node [above, thin] at (1,-5) {$0$};
    \node [above, thin] at (2,-5) {$0$};
    \node [above, thin] at (4,-5) {$0$};
    \node [above, thin] at (5,-5) {$\alpha_{s,s}$};
    \node [above, thin] at (6,-5) {$\alpha_{s,s+1}$};
    \node [above, thin] at (8,-5) {$\alpha_{s,n-1}$};
    \node [above, thin] at (9,-5) {$\alpha_{s,n}$};
    
    \foreach \x in {1,...,9}{
    \node [above, thin] at (\x,-3) {$\ldots$};
    }
    \foreach \y in {-5,...,0}{
     \node [above, thin] at (3,\y) {$\ldots$};
     \node [above, thin] at (7,\y) {$\ldots$};
    }

    \foreach \x in {1,...,2}{
    \foreach \y in {-2,...,0}{
    \fill[fill=black] (\x,\y) circle (0.03 cm);
    }}
    
    \foreach \x in {1,...,2}{
    \foreach \y in {-5,...,-4}{
    \fill[fill=black] (\x,\y) circle (0.03 cm);
    }}
    
    \foreach \x in {4,...,6}{
    \foreach \y in {-2,...,0}{
    \fill[fill=black] (\x,\y) circle (0.03 cm);
    }}
    \foreach \x in {4,...,6}{
    \foreach \y in {-5,...,-4}{
    \fill[fill=black] (\x,\y) circle (0.03 cm);
    }}
    
    \foreach \x in {8,...,9}{
    \foreach \y in {-2,...,0}{
    \fill[fill=black] (\x,\y) circle (0.03 cm);
    }}
    \foreach \x in {8,...,9}{
    \foreach \y in {-5,...,-4}{
    \fill[fill=black] (\x,\y) circle (0.03 cm);
    }}

\end{tikzpicture}
\caption{The $\alpha_{i,j}$'s connected by edges of the same color sum to the corresponding $d_k$.}
\label{cond}
\end{figure}
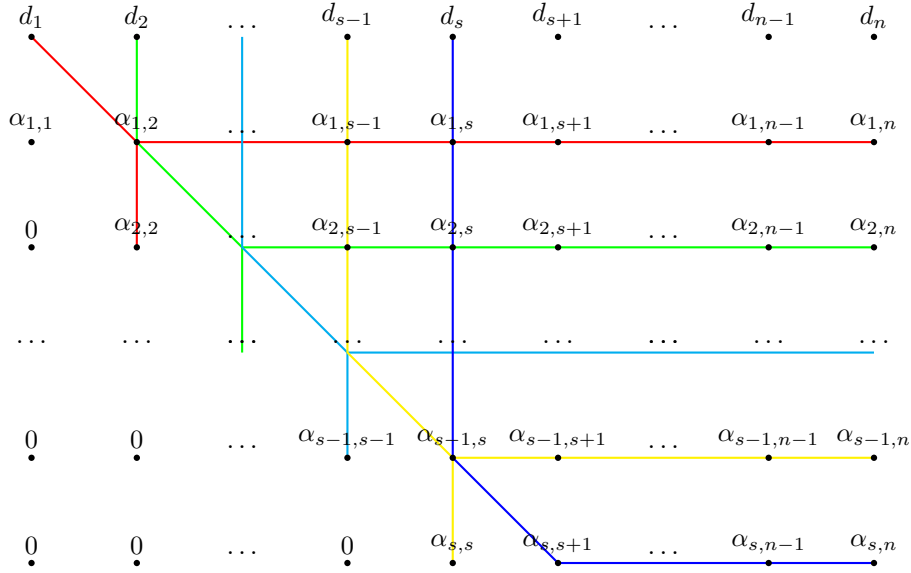

Columns $1,\ldots, s$ are called \emph{constrained} columns and all the other columns are called \emph{unconstrained}. Sometimes we will refer to $d_j$ as $\alpha_{0,j}$. 

To each $(n,s)$-table $T$ we associate a monomial ideal in the polynomial ring $\mathbb{K}[x_1,\ldots, x_n]$, where $\mathbb{K}$ denotes any field: 
$$K(T):=K_0(T)+K_1(T)+ \cdots +K_s(T),$$
where 
\begin{align*}
& K_0(T):=(x_1^{d_1},x_2^{d_2},\ldots,x_n^{d_n})=:(m_{0,1},\ldots, m_{0,n}), \\
& K_1(T):=x_1^{d_1-\alpha_{1,1}}(x_2^{d_2-\alpha_{1,2}},\ldots,x_n^{d_n-\alpha_{1,n}})=:(m_{1,2},\ldots, m_{1,n}), \\
& K_2(T):=x_1^{d_1-\alpha_{1,1}}x_2^{d_2-\alpha_{1,2}-\alpha_{2,2}}(x_3^{d_3-\alpha_{1,3}-\alpha_{2,3}},\ldots,x_n^{d_n-\alpha_{1,n}-\alpha_{2,n}})=:(m_{2,3},\ldots, m_{2,n}), \\
& \vdots \\
& K_s(T):=x_1^{d_1-\alpha_{1,1}}\cdots x_s^{d_s-\alpha_{1,s}- \cdots -\alpha_{s,s}}(x_{s+1}^{d_{s+1}-\alpha_{1,s+1}- \cdots -\alpha_{s,s+1}},\ldots,x_n^{d_n-\alpha_{1,n}- \cdots -\alpha_{s,n}})=:(m_{s,s+1},\ldots, m_{s,n}).
\end{align*}

\begin{remark}
Given a table $T$, there is a bijection between $\alpha_{i,j}$ with $0\le i\le s$, $i<j\le n$ and the generators $m_{i,j}$ of its corresponding ideal $K(T)$. Generally,
$$m_{i,j}=x_1^{d_1-\alpha_{1,1}}\cdots x_i^{d_i-\alpha_{1,i}- \cdots -\alpha_{i,i}}(x_{j}^{d_{j}-\alpha_{1,j}-\cdots -\alpha_{i,j}})=\prod_{t\in\{1,\ldots,i\}\cup\{j\}}{x_t^{d_t-\alpha_{1,t}-\cdots-\alpha_{i,t}}}.$$ 
Note that these generators need not be minimal. Such $\alpha_{i,j}$ will be called the \emph{generating} alphas. All the $\alpha_{i,i}$ will be called the \emph{diagonal} alphas. These two sets are disjoint, and all other entries of a table are zeroes.

\end{remark}

\begin{ex}($1$ condition, $3$ variables)
\label{3vars}
Consider the table $T$ in \Cref{fig:31}. Here we have $K_0(T)=(x_1^4,x_2^3,x_3^3)$, $K_1(T)=x_1(x_2,x_3)$. Thus $K(T)=K_0(T)+K_1(T)=(x_1^4,x_2^3,x_3^3,x_1x_2,x_1x_3)$.

\begin{figure}[H]
\centering
\begin{tikzpicture}[scale=2]
\draw[red,thick]  (2,-1)--(3,-1);
\draw[red,thick] (2,-1)--(1,0);  

\node [above, thin] at (1,0) {$4$};
\node [above, thin] at (2,0) {$3$};
\node [above, thin] at (3,0) {$3$};
\node [above, thin] at (1,-1) {$3$};
\node [above, thin] at (2,-1) {$2$};
\node [above, thin] at (3,-1) {$2$};
\foreach \x in {1,...,3}{
    \foreach \y in {-1,...,0}{
    \fill[fill=black] (\x,\y) circle (0.03 cm);
    }}

\end{tikzpicture}
\caption{A $(1,3)$-table: $1$ condition, $3$ variables} 
\label{fig:31}
\end{figure}
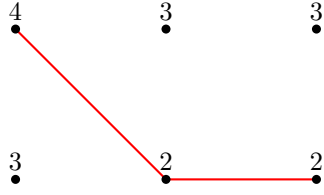
\end{ex}

\begin{ex}
\label{miro-roig}

As mentioned in the beginning of this section, in \cite{guerrieri2019lefschetz}, the author described Gorenstein codimension three algebras associated to the Ap\'ery sets of numerical semigroups. Later, in \cite{miro2020weak}, the authors explored the Lefschetz properties of a subclass of such algebras via passing to the initial ideal. This initial ideal is of the form $\mathbb{K}[x,y,z]/K$, where

$$K=(x^a,y^b,z^c, x^{a-b+\gamma}y^{b-\beta},y^{b-\beta}z^{c-\gamma}),$$
where numbers $a,b,c,\beta,\gamma$ depend on the semigroup and satisfy $1\le \beta \le b-1$, $\max\{{1,b-a+1}\}\le \gamma\le\min\{{b-1,c-1}\}$ and some other conditions. Let $\alpha:=b-\gamma$. Then we can rewrite 
$$K=(x^a,y^b,z^c, x^{a-\alpha}y^{b-\beta},y^{b-\beta}z^{c-\gamma}).$$ This is a table ideal coming from the table below:

\begin{figure}[H]
\centering
\begin{tikzpicture}[scale=2]
\draw[red,thick]  (2,-1)--(3,-1);
\draw[red,thick] (2,-1)--(1,0);

\node [above, thin] at (1,0) {$b$};
\node [above, thin] at (2,0) {$c$};
\node [above, thin] at (3,0) {$a$};
\node [above, thin] at (1,-1) {$\beta$};
\node [above, thin] at (2,-1) {$\gamma$};
\node [above, thin] at (3,-1) {$\alpha$};

\node [above, thin] at (1,0.5) {$y$};
\node [above, thin] at (2,0.5) {$z$};
\node [above, thin] at (3,0.5) {$x$};

\foreach \x in {1,...,3}{
    \foreach \y in {-1,...,0}{
    \fill[fill=black] (\x,\y) circle (0.03 cm);
    }}
  
\end{tikzpicture}
\caption{}
\end{figure}

Note that $1\le \beta \le b-1$ in particular implies $0\le \beta \le b$; $\max\{{1,b-a+1}\}\le \gamma\le\min\{{b-1,c-1}\}$ in particular implies $1\le \gamma \le c-1$, which implies $0\le \gamma \le c$. Finally, $\max\{{1,b-a+1}\}\le \gamma\le\min\{{b-1,c-1}\}$ in particular implies  $b-a+1\le \gamma\le b-1$, or, equivalently, $-a+1\le \gamma-b\le -1$, that is, $1\le \alpha\le a-1$, which in particular implies $0\le \alpha \le a$. Together with $\alpha+\gamma=b$, all of the above implies that $K$ is a table ideal. Also note that the columns are labelled to indicate the non-default order of variables.
\end{ex}




  
\begin{ex}
\label{difftables}
Consider the three tables given in Figure \ref{table_fig4}. For the first table $T_1$ we get $K_0(T_1)=(x_1^{12},x_2^7,x_3^5,x_4^4)$, $K_1(T_1)=x_1^9(x_2^3,x_3^2,x_4^2)$, $K_2(T_1)=x_1^9x_2^0(x_3^0,x_4^1)=(x_1^9)$.
Therefore, $K(T_1)=K_1(T_1)+K_2(T_1)+K_3(T_1)=(x_1^9,x_2^7,x_3^5,x_4^4)$. Note that tables $T_2$ and $T_3$ also give the same ideal: $$K(T_1)=K(T_2)=K(T_3)=(x_1^9,x_2^7,x_3^5,x_4^4).$$

\begin{figure}[ht]
\begin{center}
\begin{tikzpicture}[scale=1.4]

\draw[red,thick]  (2,-2)--(2,-1)--(4,-1);
\draw[red,thick] (2,-1)--(1,0);  

\draw[green,thick]  (2,-1)--(3,-2)--(4,-2);
\draw[green,thick] (2,-1)--(2,0); 

\draw[red,thick]  (6,-1)--(8,-1);
\draw[red,thick] (6,-1)--(5,0); 

\draw[dashed] (4.5,0.5)-- (4.5,-2.5);
\draw[dashed] (8.5,0.5)-- (8.5,-2.5);

\node [above, thin] at (1,0) {$12$};
\node [above, thin] at (2,0) {$7$};
\node [above, thin] at (3,0) {$5$};
\node [above, thin] at (4,0) {$4$};
\node [above, thin] at (1,-1) {$3$};
\node [above, thin] at (2,-1) {$4$};
\node [above, thin] at (3,-1) {$3$};
\node [above, thin] at (4,-1) {$2$};
\node [above, thin] at (1,-2) {$0$};
\node [above, thin] at (2,-2) {$3$};
\node [above, thin] at (3,-2) {$2$};
\node [above, thin] at (4,-2) {$1$};

\node [above, thin] at (5,0) {$12$};
\node [above, thin] at (6,0) {$7$};
\node [above, thin] at (7,0) {$5$};
\node [above, thin] at (8,0) {$4$};
\node [above, thin] at (5,-1) {$3$};
\node [above, thin] at (6,-1) {$7$};
\node [above, thin] at (7,-1) {$3$};
\node [above, thin] at (8,-1) {$2$};

\node [above, thin] at (9,0) {$9$};
\node [above, thin] at (10,0) {$7$};
\node [above, thin] at (11,0) {$5$};
\node [above, thin] at (12,0) {$4$};

    \foreach \x in {1,...,4}{
    \foreach \y in {-2,...,0}{
    \fill[fill=black] (\x,\y) circle (0.03 cm);
    }}
    \foreach \x in {5,...,8}{
    \foreach \y in {-1,...,0}{
    \fill[fill=black] (\x,\y) circle (0.03 cm);
    }}
    \foreach \x in {9,...,12}{
    \foreach \y in {0,...,0}{
    \fill[fill=black] (\x,\y) circle (0.03 cm);
    }}

\end{tikzpicture}
\end{center}
\caption{Three different tables representing the same ideal $K=(x_1^9,x_2^7,x_3^5,x_4^4).$}
\label{table_fig4}
\end{figure}
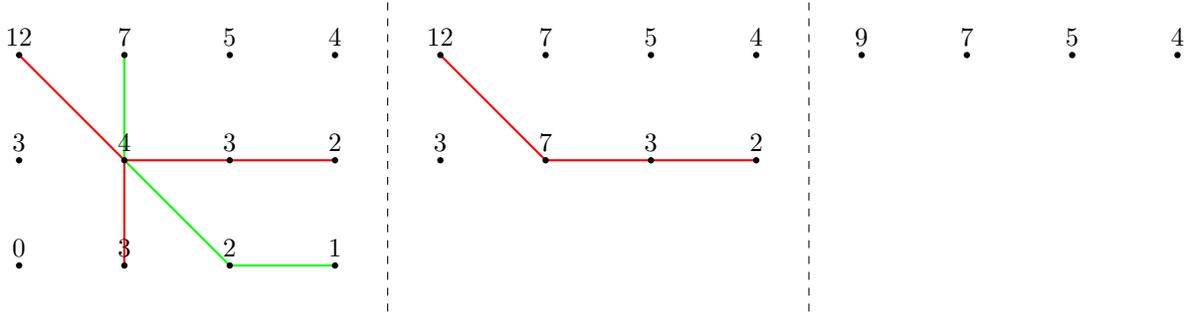
\end{ex}

\begin{definition}
A \emph{generalised table} is a disjoint union of tables in disjoint sets of column indices (variables). For example, a $(0,n)$-table is a disjoint union of $n$ $(0,1)$-tables. To a generalised table we associate the sum of ideals associated to the individual tables inside the polynomial ring over the (disjoint) union of variables appearing in all the tables.
\end{definition}

\begin{ex}
\label{tableunion}
Consider the disjoint union of tables $T=T_1\sqcup T_2\sqcup T_3$ in the Figure \ref{table_fig5}. The corresponding ideal in $\mathbb{K}[x_1,\ldots,x_8]$ is $$K(T)=K(T_1)+K(T_2)+K(T_3)=(x_2^{12},x_3^{9},x_5^{6},x_7^{4}, x_2^{9}(x_3^3,x_5^3,x_7^3),x_2^9x_3(x_5,x_7^2), x_4^3, x_1^8,x_6^6,x_8^7,x_1^3(x_6^3,x_8^2))$$

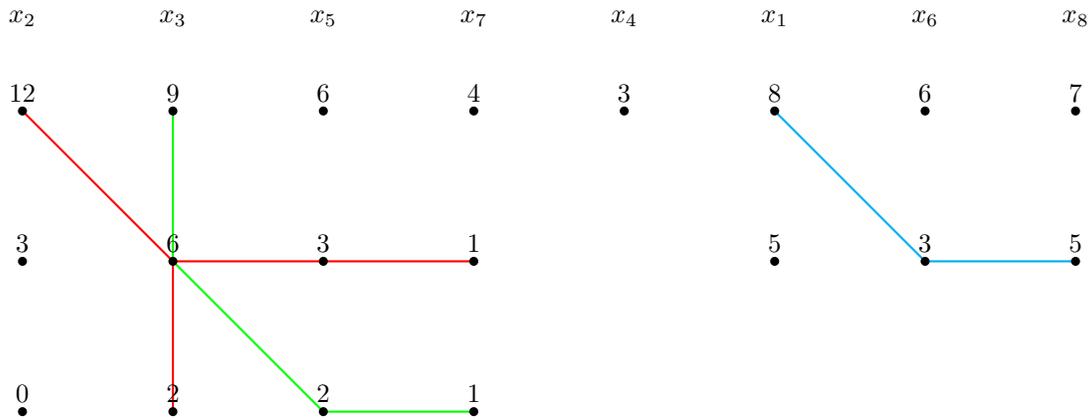
\begin{figure}[H]
\centering
\begin{tikzpicture}[scale=2]

\draw[red,thick]  (2,-2)--(2,-1)--(4,-1);
\draw[red,thick] (2,-1)--(1,0);  

\draw[green,thick]  (2,-1)--(3,-2)--(4,-2);
\draw[green,thick] (2,-1)--(2,0); 

\draw[cyan,thick]  (7,-1)--(8,-1);
\draw[cyan,thick] (7,-1)--(6,0);

\node [above, thin] at (1,0.5) {$x_2$};
\node [above, thin] at (2,0.5) {$x_3$};
\node [above, thin] at (3,0.5) {$x_5$};
\node [above, thin] at (4,0.5) {$x_7$};
\node [above, thin] at (5,0.5) {$x_4$};
\node [above, thin] at (6,0.5) {$x_1$};
\node [above, thin] at (7,0.5) {$x_6$};
\node [above, thin] at (8,0.5) {$x_8$};
\node [above, thin] at (1,0) {$12$};
\node [above, thin] at (2,0) {$9$};
\node [above, thin] at (3,0) {$6$};
\node [above, thin] at (4,0) {$4$};
\node [above, thin] at (1,-1) {$3$};
\node [above, thin] at (2,-1) {$6$};
\node [above, thin] at (3,-1) {$3$};
\node [above, thin] at (4,-1) {$1$};
\node [above, thin] at (1,-2) {$0$};
\node [above, thin] at (2,-2) {$2$};
\node [above, thin] at (3,-2) {$2$};
\node [above, thin] at (4,-2) {$1$};

\node [above, thin] at (5,0) {$3$};
\node [above, thin] at (6,0) {$8$};
\node [above, thin] at (7,0) {$6$};
\node [above, thin] at (8,0) {$7$};
\node [above, thin] at (6,-1) {$5$};
\node [above, thin] at (7,-1) {$3$};
\node [above, thin] at (8,-1) {$5$};

    \foreach \x in {1,...,4}{
    \foreach \y in {-2,...,0}{
    \fill[fill=black] (\x,\y) circle (0.03 cm);
    }}
    \fill[fill=black] (5,0) circle (0.03 cm);
    \foreach \x in {6,...,8}{
    \foreach \y in {-1,...,0}{
    \fill[fill=black] (\x,\y) circle (0.03 cm);
    }}

\end{tikzpicture}
\caption{A disjoint union of three tables}
\label{table_fig5}
\end{figure}

\end{ex}

\begin{remark}
\label{permute}
Within any table, unconstrained columns can be permuted. The resulting matrix is still a table and the corresponding ideal does not change. When permuting unconstrained columns, it is important to permute them together with the corresponding labels in order to keep track of which column corresponds to which variable. Individual tables can of course be permuted with each other. We do not distinguish between generalised tables which are equal up to these permutations.
\end{remark}

\begin{definition}
An ideal $K\in \mathbb{K}[x_1,\ldots, x_n]$ will be called a \emph{table ideal} if there exists a generalised table $T$ in variables $x_1,\ldots, x_n$ such that $K=K(T)$.
\end{definition}

\begin{definition}
A generalised table $T$ in variables $x_1,\ldots, x_n$ is called \emph{proper} if $K(T)\not=\mathbb{K}[x_1,\ldots, x_n]$.
\end{definition}

\begin{remark}
Note that if $K$ is a proper table ideal, then $\mathbb{K}[x_1,\ldots, x_n]/K$ is Artinian.
\end{remark}

\begin{definition}
Let $\mathbb{K}$ be a field and let $A =\bigoplus_{i=0}^{i=c} A_i$ be a standard graded Artinian
$\mathbb{K}$-algebra with $A_c\not=(0)$.
We say that $A$ has the \emph{strong Lefschetz property (SLP)} if there exists a linear form $\ell\in A_1$ such that the multiplication map
$$\times \ell^d\colon A_i\mapsto A_{d+i}$$ 
has full rank for all $1\le d\le c-1$ and $0\le i \le c-d$. If $\dim (A_i)=\dim(A_{c-i})$ for all $i=0,1,\ldots,c$, we say that $A$ has a \emph{symmetric Hilbert series}. We say that 
$A$ has the \emph{SLP in the narrow sense} if $A$ has the SLP and its Hilbert series is symmetric. 
\end{definition}
For more details on Lefschetz properties, see \cite{SLP}

\begin{theorem}[Theorem~3 in \cite{table}]\label{tablethm}
Let $T$ be a proper $($single$)$ table in variables $x_1,\ldots, x_n$ and let $\mathbb{K}$ be a field of characteristic zero. Then $\mathbb{K}[x_1,\ldots, x_n]/K(T)$ has the SLP in the narrow sense,
and its maximal socle degree is $d_1+\cdots+d_n -\alpha_{1,1}-n$ $($where we set $\alpha_{1,1}:=0$ if $s=0)$.
\end{theorem}
From \Cref{tablethm} and Theorem~3.34 in \cite{SLP} we conclude the following.

\begin{corollary}
Let $K$ be a proper table ideal and $\mathbb{K}$ be a field of characteristic zero. Then $\mathbb{K}[x_1,\ldots, x_n]/K$ has the SLP in the narrow sense. The maximal socle degree is the sum of maximal socle degrees of algebras coming from individual tables.
\end{corollary}

\subsection{Tables in a normal form}
\label{sec:new_results}

\begin{definition}
A proper $(s,n)$-table $T$ with $s\ge 1$ is said to be \emph{in a normal form} if the following conditions hold: 
\begin{enumerate}
    \item $\alpha_{i,i}\neq 0$ for all $i=1,\ldots, s$;
    \item for all $j\ge s+1$, at least one of the numbers $\alpha_{1,j},\ldots,\alpha_{s,j}$ is different from zero. If this condition fails, we will say that the corresponding column is an \emph{unconstrained almost zero column} $($note that $\alpha_{0,j}=d_j$ has to be nonzero, otherwise the table is improper; this justifies the word "almost"$)$. Note that if we have a constrained almost zero column $j$, $1\le j\le s$, then $\alpha_{1,j}=\cdots=\alpha_{s,j}=0$. In particular, $\alpha_{j,j}=0$, thus the table fails the previous condition anyway;
    \item $\sum_{i=1}^{s}{\alpha_{i,j}}<d_j$ for all $j=1,\ldots, n$. Note that the definition of a table does allow for non-strict inequalities. So, if an equality occurs in some column $j$, we will call this a \emph{singularity} in column $j$.
\end{enumerate}
Any proper $(0,n)$-table is considered to be in a normal form.
\end{definition}
\begin{definition}
A proper generalised table is said to be \emph{in a normal form} if it is a disjoint union of (proper) tables, each of which is in a normal form.
\end{definition}
Consider \Cref{difftables}. The first two (generalised) tables are not in a normal form, whereas the third (generalised) table is in a normal form. In \Cref{reduction} we show how to reduce any generalised table to a normal form without changing the corresponding ideal.
\begin{remark}

\label{rmksupp}
Given a proper table in a normal form, $\supp(m_{i,j})=\{1,\ldots, i\}\cup \{j\}$, since no variable outside of this set divides $m_{i,j}$. 
Any variable in this set divides $m_{i,j}$. Indeed, assume it is not the case for some $t\in \{1,\ldots, i\}\cup\{j\}$. Then $d_t=\alpha_{1,t}+\cdots+\alpha_{i,t}$. Note that for $i=0$ this implies  $t=j$ and $d_j=0$ and thus the table is improper. Otherwise, if $i\ge 1$ (implying $s\ge 1$), we get $d_t=\alpha_{1,t}+\cdots+\alpha_{i,t}\le \alpha_{1,t}+\cdots+\alpha_{s,t}\le d_t$, where the last inequality comes from the definition of a table. This implies $d_t=\alpha_{1,t}+\cdots+\alpha_{s,t}$ and thus there is a singularity in column $t$, which means that the table is not in a normal form.
\end{remark}

\begin{lemma}\label{minimalgens}
Consider a proper table $T$ in a normal form and let $K(T)$ be the corresponding ideal. Then for any $(i,j)$ with $0\le i\le s$, $i<j\le n$ we have the following: $\alpha_{i,j}\not=0$ if and only if $m_{i,j}$ is a minimal generator of $K(T)$.
\end{lemma}

\begin{proof}
$\Leftarrow$ Assume $\alpha_{i,j}=0$ for some $(i,j)$ with $i<j$. If $i=0$, that is, if $\alpha_{0,j}=d_j=0$, the table is improper, thus we will assume $i\ge 1$. Then
\begin{align*}
m_{i,j} & =\prod_{t\in\{1,\ldots,i\}\cup\{j\}}{x_t^{d_t-\alpha_{1,t}-\cdots-\alpha_{i,t}}} \\
& =\prod_{t\in\{1,\ldots,i-1\}}{x_t^{d_t-\alpha_{1,t}-\cdots-\alpha_{i-1,t}-(\alpha_{i,t}=0)}}x_i^{d_i-\alpha_{1,i}-\cdots-\alpha_{i,i}}x_j^{d_j-\alpha_{1,j}-\cdots-\alpha_{i-1,j}-(\alpha_{i,j}=0)} \\
& =x_i^{d_i-\alpha_{1,i}-\cdots-\alpha_{i,i}}\prod_{t\in\{1,\ldots,i-1\}\cup\{j\}}{x_t^{d_t-\alpha_{1,t}-\cdots-\alpha_{i-1,t}}}    
\end{align*}
Here $\alpha_{i,t}=0$ from the definition of a table (since $i>t$) and $\alpha_{i,j}=0$ by assumption. Given that
$$m_{i-1,j}=\prod_{t\in\{1,\ldots,i-1\}\cup\{j\}}{x_t^{d_t-\alpha_{1,t}-\cdots-\alpha_{i-1,t}}},$$
we conclude that $m_{i-1,j}$ divides $m_{i,j}$.

$\Rightarrow$ Let $\alpha_{i,j}\not=0$ for some $(i,j)$ with $i<j$. Assume $m_{i,j}$ is not a minimal generator of $K(T)$, that is, there exists some $m_{p,q}$ with $0\le p\le s$, $p<q\le n$ and $(p,q)\not=(i,j)$ which divides $m_{i,j}$. Then $\supp(m_{p,q})=\{1,\ldots, p\}\cup \{q\}\subseteq\{1,\ldots, i\}\cup\{j\}=\supp(m_{i,j})$. This can happen in $2$ cases:
\begin{itemize}
    \item[(1)] $q=j$ and $p<i$,
    \item[(2)] $q\le i$.
\end{itemize}
Note that for $i=0$ none of these cases are possible and we are done. Thus we will assume $i\ge 1$ (implying $s\ge 1$).
Consider the first case. Then $(p,q)=(i-k,j)$ for some $k>0$. Compare $x_j$-exponents in $m_{i,j}$ and $m_{i-k,j}$.
In $m_{i,j}$ the $x_j$-exponent is $d_j-\alpha_{1,j}-\cdots-\alpha_{i,j}$.
In $m_{i-k,j}$ the $x_j$-exponent is
$d_j-\alpha_{1,j}-\cdots-\alpha_{i-k,j}$.

Clearly, the second number is strictly larger than the first one (and thus $m_{i-k,j}$ can not divide $m_{i,j}$), unless $\alpha_{i-k+1,j}=\alpha_{i-k+2,j}=\ldots=\alpha_{i,j}=0$. But $\alpha_{i,j}=0$ contradicts the assumption of this lemma. 


Now consider the second case, that is, assume that $m_{i,j}$ is divisible by $m_{p,q}$ with $p<q\le i<j$. We will compare the exponents of $x_q$ in these monomials.
In $m_{i,j}$ the exponent is $d_q-\alpha_{1,q}-\dots-\alpha_{i,q}=d_q-\alpha_{1,q}-\dots-\alpha_{q,q}$.
In $m_{p,q}$ the exponent is $d_q-\alpha_{1,q}-\dots-\alpha_{p,q}$.
Clearly, the second number is strictly larger than the first one (and thus $m_{p,q}$ can not divide $m_{i,j}$), unless $\alpha_{p+1,q}=\alpha_{p+2,q}=\ldots=\alpha_{q,q}=0$. But $\alpha_{q,q}=0$ implies that the table is not in a normal form. 
\end{proof}

\subsection{Table ideals and simplicial complexes}
\label{sec:simplices}
Let $K$ be a proper monomial ideal in $\mathbb{K}[x_1,\ldots,x_n]$. We associate to $K$ a weighted simplicial complex $S(K)$ as follows. First we generate a simplicial complex for all $x_{i_1}\ldots x_{i_t}$ such that $\{i_1,\ldots,i_t\}$ is the support of some minimal generator of $K$. To each face we assign a non-negative integer weight which is exactly the number of minimal generators of $K$ with the corresponding support.
\begin{example}
Consider the ideal from \Cref{tableunion}, that is,
$$K(T)=K(T_1)+K(T_2)+K(T_3)=(x_2^{12},x_3^{9},x_5^{6},x_7^{4}, x_2^{9}(x_3^3,x_5^3,x_7^3),x_2^9x_3(x_5,x_7^2), x_4^3, x_1^8,x_6^6,x_8^7,x_1^3(x_6^3,x_8^2)).$$
The simplicial complex is described below. All faces, except those marked with zeros, have weight $1$.

\begin{figure}[H]
\centering
\begin{tikzpicture}[scale=2]

\filldraw [fill=cyan, fill opacity=0.5] (1, 0) -- (2, 0) -- (2,-1) -- (1, 0) -- (1,-1) -- (2, -1);
\draw [thick] (1, 0) -- (2, 0) -- (2,-1) -- (1, 0) -- (1,-1) -- (2, -1);
\draw [thick] (4, 0) -- (4,-1) -- (5,-1);

\node [above, left, thin] at (1,0) {$x_2$};
\node [above, right,thin] at (2,0) {$x_7$};
\node [below, left, thin] at (1,-1) {$x_5$};
\node [below, right,thin] at (2,-1) {$x_3$};

\node [below, thin] at (1.5,-1) {$0$};
\node [right, thin] at (2,-0.5) {$0$};

\node [above,right, thin] at (3,0) {$x_4$};
\node [below, left, thin] at (4,-1) {$x_1$};
\node [above, left,thin] at (4, 0) {$x_6$};
\node [below, right,thin] at (5,-1) {$x_8$};

\fill[fill=black] (1, 0) circle (0.03 cm);
\fill[fill=black] (2, 0) circle (0.03 cm);
\fill[fill=black] (3, 0) circle (0.03 cm);
\fill[fill=black] (4, 0) circle (0.03 cm);
\fill[fill=black] (1,-1) circle (0.03 cm);
\fill[fill=black] (2,-1) circle (0.03 cm);
\fill[fill=black] (4,-1) circle (0.03 cm);
\fill[fill=black] (5,-1) circle (0.03 cm);

\end{tikzpicture}
\caption{}
\end{figure}

\begin{remark}
\label{nicecomplex}
Let $T$ be a proper (single) table in a normal form. Then:
\begin{enumerate}
\item All faces of $S(K(T))$ have weights $0$ or $1$. Indeed, from \Cref{rmksupp} we know that $\supp(m_{i,j})=\{1,\ldots, i\}\cup \{j\}$. Thus different $(i,j)$ will give different supports. Face $x_1\ldots x_ix_j$ has weight $1$ if $m_{i,j}$ is a minimal generator and $0$ otherwise.
\item All vertices of $S(K(T))$ have weight $1$. Indeed, if $x_j$ has weight $0$, then $x_j^{d_j}=m_{0,j}$ is not a minimal generator of $K(T)$. Since $m_{0,j}\leftrightarrow \alpha_{0,j}=d_j$, this, by \Cref{minimalgens} implies $d_j=0$ and thus the table is improper.

\item Clearly, all facets, that is, maximal faces with respect to inclusion, have nonzero weighs (in fact, this holds for any simplicial complex obtained from a proper ideal), thus each of them has weight $1$.
\end{enumerate}
\end{remark}
\begin{remark}
\label{infotable}
Let $T$ be a proper (single) table in a normal form. Then:
\begin{enumerate}
\item The number of vertices ($0$-faces) of $S(K(T))$ equals the number of variables (columns) of $T$. This is quite straightforward.
\item The dimension of $S(K(T))$ equals the number of colours in $T$. Indeed, let $s-1$ be the maximal face dimension. First of all, since there is a face of dimension $s-1$, it comes from a minimal generator $m_{i,j}$ whose support has cardinality $s$. Since $\supp(m_{i,j})=\left\{1,\ldots,i\right\}\cup\left\{j\right\}$, we conclude that $i=s-1$. Since index $i$ is less than or equal to the number of colours of $T$, we conclude that $T$ has at least $s-1$ colours. Now assume that $T$ has at least $s$ colours. We do not have minimal generators $m_{s,s+1}, \ldots, m_{s,n}$ (otherwise we would have an $s$-dimensional face), which by \Cref{minimalgens} implies $\alpha_{s,s+1}=\ldots=\alpha_{s,n}=0$. But this is impossible: condition $s$ then says $d_s=\sum_{i=1}^{s-1}{\alpha_{i,s}}+\sum_{j=s+1}^{n}{\alpha_{s,j}}=\sum_{i=1}^{s-1}{\alpha_{i,s}}$. Together with this, according to the definition of a table, we have $\sum_{i=1}^{s}{\alpha_{i,s}}\le d_s$, which implies $\alpha_{s,s}=0$ and thus $T$ is not in a normal form.
\item Let $x_{i_1},\ldots, x_{i_s}$ be the order of constrained columns in $T$ with $s$ colours. Then $x_{i_k}$ is contained in the intersection of all faces with weights $1$ of $S(K(T))$ whose dimensions are greater than or equal to $k$. In other words, the first constrained variable is contained in all faces with weights $1$, except possibly vertices, the second constrained variable is contained in all faces with weights $1$, except possibly vertices and edges and so on. 
\end{enumerate}
\end{remark}

\end{example}
\begin{definition}
We call a weighted simplicial complex \emph{connected} if so is the underlying unweighted simplicial complex.
\end{definition}
\begin{definition}
A generalised table $T$ is said to have $k$ connected components if it is a disjoint union of $k$ tables. Each $(0,n)$-table counts as $n$ tables. In other words, the number of connected components of $T$ is equal to the number of connected components of this table treated as a simple graph. 
\end{definition}
\begin{theorem}
\label{complexthm}
Let $T$ be a proper generalised table in a normal form with $k$ connected components. Then $S(K(T))$ has $k$ connected components. 
\end{theorem}
\begin{proof}
Note that variables appearing in different connected components of a proper generalised table in a normal form will end up in different connected components in the corresponding simplicial complex. Therefore, it is enough to prove this theorem for $k=1$, that is, for proper tables in a normal form with $1$ connected component.
If $s=0$, the table is a (nonzero) singleton and the statement is trivial.  
Let $T$ be a proper table in a normal form with $s\ge1$ colours. Note that by definition $T$ has at least $1$ unconstrained column. Also note that numbers $\alpha_{s,s+1},\ldots, \alpha_{s,n}$ can not all be equal to $0$. Indeed, assume the contrary. Then condition $s$ says $d_s=\sum_{i=1}^{s-1}{\alpha_{i,s}}+\sum_{j=s+1}^{n}{\alpha_{s,j}}=\sum_{i=1}^{s-1}{\alpha_{i,s}}$. Together with this, according to the definition of a table, we have $\sum_{i=1}^{s}{\alpha_{i,s}}\le d_s$, which implies $\alpha_{s,s}=0$ and thus $T$ is not in a normal form. Assume without loss of generality that $\alpha_{s,s+1}\alpha_{s,s+2}\cdots\alpha_{s,r}\not=0$ and $\alpha_{s,r+1}=\dots=\alpha_{s,n}=0$ with $r-s\ge 1$ and $n-r\ge 0$. From \Cref{minimalgens} we conclude that $m_{s,s+1},\ldots, m_{s,r}$ are minimal generators of $K(T)$. Recall that $m_{i,j}$ is a monomial in variables $x_1,\ldots, x_i,x_j$. Therefore, variables $x_1,\ldots,x_r$ belong to the same connected component of $S(K(T))$. If $n=r$, we are done. Otherwise, consider any variable $x_{t}$, $r+1\le t\le n$. Since $\alpha_{s,t}=0$, there is some $\alpha_{i,t}\not=0$ with $0<i<s$. Indeed, if it is not the case, then $T$ is not in a normal form since it has an almost zero column. Then $m_{i,t}$ is a minimal generator in variables $x_1,\ldots, x_i, x_{t}$. Since $i\not=0$, $x_{t}$ is in the same connected component as $x_1$ (and therefore $x_2,\ldots, x_r$).
\end{proof}

\subsection{Uniqueness of a generalised table in a normal form for a table ideal}
\label{sec:uniqueness}
We have seen that different generalised tables can result into the same ideal. If, however, we restrict to generalised tables in a normal form, the answer is not so obvious anymore. Recall that we do not distinguish between generalised tables which are equal up to a permutation of tables or a permutation of unconstrained columns. Note that we can restrict to table ideals $K$ such that $S(K)$ is connected. We can also restrict to the situation of $S(K)$ having dimension $s\ge 1$. Therefore, for the rest of this subsection we will assume $K=K(T)$ with $T$ in a normal form, one connected component and $s\ge 1$ colours.

\begin{definition}
Let $K$ be a table ideal such that $S(K)$ is connected. Let $x_i$ be a variable such that all minimal generators of $K$, except $x_1^{d_1},\ldots, x_n^{d_n}$, have the same $x_i$-exponents. We call $x_i$ a $\emph{possible first} \emph{variable}$. The non-empty set of all such $x_i$ will be denoted by $F(K)$.
\end{definition}

\begin{lemma}
\label{lemma: unique first}
Let $K$ be a table ideal such that $S(K)$ is connected and has dimension $s\ge 1$. Then there exists a unique variable $x_i$ which can possibly be the first constrained variable.
\end{lemma}

\begin{proof}
Clearly, any variable that can possibly start a table for $K$ is contained in $F(K)$. If $F(K)$ is a singleton, we are done. Now assume $F(K)$ has at least $2$ elements. Let us try to start constructing a table from $x_1$ (assume without loss of generality that $x_1\in F(K)$). First of all, we know that among numbers $\alpha_{1,2}, \ldots, \alpha_{1,n}$ there is at most one nonzero number. Indeed, if $\alpha_{1,i}\not=0\not=\alpha_{1,j}$ for some $i\not=j$, then monomials $m_{1,i}=x_1^{d_1-\alpha_{1,1}}x_i^{d_1-\alpha_{1,i}}$ and $m_{1,j}=x_1^{d_1-\alpha_{1,1}}x_j^{d_1-\alpha_{1,j}}$ are both minimal generators of $K$. Already from these $2$ monomials we see that $x_1$ is the only possible first variable. We conclude that row $1$ contains at most one nonzero entry except $\alpha_{1,1}$. Consider the following cases:
\begin{enumerate}
    \item Row $1$ has exactly $1$ nonzero number except $\alpha_{1,1}$, for example $\alpha_{1,j}\not=0$. Note that $\alpha_{1,j}\leftrightarrow m_{1,j}=x_1^{d_1-\alpha_{1,1}}x_j^{d_1-\alpha_{1,j}}$. Then $x_j\in F(K)$, since we assumed that $F(K)$ has at least $2$ elements. Note that $x_j$ can not be a constrained variable. Indeed, assume the opposite. Note that we have at least $1$ minimal generator $m$ whose support consists of $s+1$ variables: $s$ constrained ones and $1$ unconstrained. If $x_j$ is constrained, then the $x_j$-exponent of $m$ is $d_j-\alpha_{1,j}-\ldots-\alpha_{j,j}<d_1-\alpha_{1,j}$, since $\alpha_{j,j}>0$ due to normality of our future table. This implies that $x_j\not\in F(K)$, a contradiction. Therefore, $x_j$ is not constrained. Note that the only monomials containing the unconstrained variable $x_j$ are monomials coming from column $j$. However, all such monomials will have strictly smaller $x_j$-exponents than $d_1-\alpha_{1,j}$. Therefore, all numbers below $\alpha_{1,j}$ (if they exist at all) have to be $0$. Monomials coming from other columns do not contain variable $x_j$ at all and therefore such minimal generators should not exist. This implies that the only minimal generator except pure powers is $m_{1,j}$. This implies that we only have $2$ variables (otherwise $S(K)$ is disconnected) and $1$ colour. We can then set $j=2$ and then the table is   
    \begin{figure}[H]
\centering
\begin{tikzpicture}[scale=2]
\draw[red,thick] (2,-1)--(1,0);  

\node [above, thin] at (1,0) {$d_1$};
\node [above, thin] at (2,0) {$d_2$};

\node [above, thin] at (1,-1) {$\alpha_{1,1}$};
\node [above, thin] at (2,-1) {$d_1$};
\foreach \x in {1,...,2}{
    \foreach \y in {-1,...,0}{
    \fill[fill=black] (\x,\y) circle (0.03 cm);
    }}

\end{tikzpicture}
\caption{}
\end{figure}
Note that in this case we have $F(K)=\{x_1,x_2\}$ and $d_1<d_2$.    
\item Row $1$ does not have nonzero numbers except $\alpha_{1,1}$. Then we have $s\ge 2$, therefore, at least one more constrained variable. Note that whichever variable comes after $x_1$, it has to belong to $F(K)$. Assume $x_2$ is the next constrained variable (again, we assume without loss of generality that $x_2\in F(K)$). Then our future table starts like this, with possibly more colours and constrained variables:

\begin{figure}[H]
\centering
\begin{tikzpicture}[scale=2]
\draw [red, thick] (1,0) -- (2,-1) -- (2,-2) -- (2,-1) -- (5,-1);
\draw [green, thick] (2,0) -- (2,-1) -- (3,-2) -- (5,-2);
 
\node [above, thin] at (1,0) {$d_1$};
\node [above, thin] at (2,0) {$d_2$};
\node [above, thin] at (3,0) {$\ldots$};
\node [above, thin] at (4,0) {$d_{n-1}$};
\node [above, thin] at (5,0) {$d_n$};
\node [above, thin] at (1,-1) {$\alpha_{1,1}$};
\node [above, thin] at (2,-1) {$0$};
\node [above, thin] at (3,-1) {$\ldots$};
\node [above, thin] at (4,-1) {$0$};
\node [above, thin] at (5,-1) {$0$};
\node [above, thin] at (1,-2) {$0$};
\node [above, thin] at (2,-2) {$d_1$};
\node [above, thin] at (3,-2) {$\ldots$};
\node [above, thin] at (4,-2) {$\alpha_{2,n-1}$};
\node [above, thin] at (5,-2) {$\alpha_{2,n}$};

\foreach \x in {1,...,2}{
    \foreach \y in {-2,...,0}{
    \fill[fill=black] (\x,\y) circle (0.03 cm);
    }}
\foreach \x in {4,...,5}{
    \foreach \y in {-2,...,0}{
    \fill[fill=black] (\x,\y) circle (0.03 cm);
    }}

\end{tikzpicture}
\caption{}

\end{figure}
Note that $d_1<d_2$. Now, if $F(K)$ has $2$ elements, the proof is complete. Note that in all the cases we had $x_1<x_2$. Now we assume $F(K)$ has at least $3$ elements and proceed in a similar way: row $2$ has at most one nonzero number except $\alpha_{2,2}=d_1$. Then this case will split into $2$ subcases:
\begin{enumerate}
    \item[2.1] Row $2$ has exactly $1$ nonzero number except $\alpha_{2,2}$, say, $\alpha_{2,j}\not=0$. Similarly to Case~1, we will conclude that we only have $3$ variables and set $j=3$. The table then looks like this: 
    
\begin{figure}[H]
\centering
\begin{tikzpicture}[scale=2]

\draw[red,thick]  (2,-2)--(2,-1)--(3,-1);
\draw[red,thick] (2,-1)--(1,0);  

\draw[green,thick]  (2,-1)--(3,-2)--(3,-2);
\draw[green,thick] (2,-1)--(2,0);

\node [above, thin] at (1,0) {$d_1$};
\node [above, thin] at (2,0) {$d_2$};
\node [above, thin] at (3,0) {$d_3$};
\node [above, thin] at (1,-1) {$\alpha_{1,1}$};
\node [above, thin] at (2,-1) {$0$};
\node [above, thin] at (3,-1) {$0$};
\node [above, thin] at (1,-2) {$0$};
\node [above, thin] at (2,-2) {$d_1$};
\node [above, thin] at (3,-2) {$d_2$};

    \foreach \x in {1,...,3}{
    \foreach \y in {-2,...,0}{
    \fill[fill=black] (\x,\y) circle (0.03 cm);
    }}

\end{tikzpicture}
\caption{}
\end{figure}  
Note that in this case we have $F(K)=\{x_1,x_2,x_3\}$ and $d_1<d_2<d_3$.
\item[2.2] Row $2$ does not have nonzero numbers except $\alpha_{2,2}$. This subcase will again split into $2$ further subcases. 
\end{enumerate}
\end{enumerate}
To summarize all of the above: in all the cases above we had the following: if $F(K)=\{x_{i_1},\ldots, x_{i_r}\}$, then the first variables of any table $T$ in a normal form with $K(T)=K$ are variables $x_{i_1},\ldots, x_{i_r}$, with some permutation. Furthermore, the permutation can be uniquely determined from $d_j$: if $d_{i_1}<d_{i_2}\ldots<d_{i_r}$, then the order of variables is exactly $x_{i_1}, x_{i_2},\ldots, x_{i_r}$. In particular, we can uniquely determine the true first variable among all the possible ones.  Note that equality between, say, $d_{i_1}$ and $d_{i_2}$, where $x_{i_1},x_{i_2}\in F(K)$ and $i_1\not=i_2$, is impossible because in this case the constructed $T$ is not in a normal form and splits into a disjoint union of tables, which gives a disconnected $S(K)$.
\end{proof}
\begin{remark}
\label{rem: induction}
Let $K=K(T)$, where $T$ is a (single) table in a normal form with the default order of variables. Let $J=K:(x_1^{d_1-\alpha_{1,1}})$, where $I_1:I_2$ denotes the ideal quotient. 

Then 
\begin{multline*}
J=(x_1^{\alpha_{1,1}},x_2^{d_2-\alpha_{2,1}},\ldots, x_n^{d_n-\alpha_{n,1}}, x_2^{d_2-a_{2,1}-a_{2,2}}(x_3^{d_3-a_{3,1}-a_{3,2}},\ldots,x_n^{d_n-a_{n,1}-a_{n,2}}),\\ \ldots, x_2^{d_2-a_{2,1}-a_{2,2}}\cdots x_s^{d_s-a_{s,1}- \cdots -a_{s,s}}(x_{s+1}^{d_{s+1}-a_{s+1,1}- \cdots -a_{s+1,s}},\ldots,x_n^{d_n-a_{n,1}- \cdots -a_{n,s}})). 
\end{multline*}
Now let $J'$ be obtained from $J$ by throwing away the first generator $x_1^{\alpha_{1,1}}$. Note that $J'$ is a table ideal in variables $x_2,\ldots, x_n$ coming from the following table:

\begin{figure}[H]

\centering
\begin{tikzpicture}[scale=1.4]
	 \draw [red,thick]     (2,-1) -- (9,-1);
      \draw[red,thick]    (2,-1) -- (2,-2);
      \draw [red,thick]  (2,-1)--(1,0);
      
      \draw[green, thick]  (2,-1)-- (3,-2) -- (9,-2);
      \draw[green, thick]  (3,-2) -- (3,-3);
      \draw[green,thick] (2,-1)--(2,0);

       \draw[cyan, thick]  (3,-1)-- (3,-2) -- (4,-3)--(9,-3);
      \draw[cyan, thick]  (4,-4)-- (4,-3);
      \draw[cyan, thick] (3,-1)--(3,0);

      \draw[yellow,thick]  (4,-1)-- (4,-3) --(5,-4)-- (9,-4);
      \draw[yellow,thick]  (5,-4)-- (5,-5);
      \draw[yellow,thick]  (4,-1)-- (4,0);
      
      \draw[blue,thick]  (5,-1)--(5,-4) --(6,-5)-- (9,-5);
      \draw[blue,thick]  (5,-1)--(5,-0);
      \draw[dashed]  (1.5,0.5)-- (1.5,-5.5);
    \foreach \x in {1,...,2}{
    \node [above, thin] at (\x,0) {$d_\x$};

    }
    \node [above, thin] at (2,-1) {$\alpha_{2,1}$};
    \node [above, thin] at (1,-1) {$\alpha_{1,1}$};
    \node [above, thin] at (1,-2) {$0$};
    \node [above, thin] at (2,-2) {$\alpha_{2,2}$};
    \node at (1.5,-5.5) {\rotatebox{90}{\ding{34}}};

    \node [above, thin] at (4,-1) {$\alpha_{s-1,1}$};
    \node [above, thin] at (5,-1) {$\alpha_{s,1}$};
    \node [above, thin] at (6,-1) {$\alpha_{s+1,1}$};
    \node [above, thin] at (8,-1) {$\alpha_{n-1,1}$};
    \node [above, thin] at (9,-1) {$\alpha_{n,1}$};
    
    \node [above, thin] at (4,-2) {$\alpha_{s-1,2}$};
    \node [above, thin] at (5,-2) {$\alpha_{s,2}$};
    \node [above, thin] at (6,-2) {$\alpha_{s+1,2}$};
    \node [above, thin] at (8,-2) {$\alpha_{n-1,2}$};
    \node [above, thin] at (9,-2) {$\alpha_{n,2}$};

    \node [above, thin] at (4,0) {$d_{s-1}$};    
    \node [above, thin] at (5,0) {$d_{s}$};  
    \node [above, thin] at (6,0) {$d_{s+1}$};  
    \node [above, thin] at (8,0) {$d_{n-1}$};
    \node [above, thin] at (9,0) {$d_{n}$};

    \node [above, thin] at (1,-4) {$0$};
    \node [above, thin] at (2,-4) {$0$};
    
    \node [above, thin] at (4,-4) {$\alpha_{s-1,s-1}$};
    
    \node [above, thin] at (5,-4) {$\alpha_{s,s-1}$};
    
    \node [above, thin] at (6,-4) {$\alpha_{s+1,s-1}$};
    \node [above, thin] at (8,-4) {$\alpha_{n-1,s-1}$};
    \node [above, thin] at (9,-4) {$\alpha_{n,s-1}$};

    \node [above, thin] at (1,-5) {$0$};
    \node [above, thin] at (2,-5) {$0$};
    \node [above, thin] at (4,-5) {$0$};
    \node [above, thin] at (5,-5) {$\alpha_{s,s}$};
    \node [above, thin] at (6,-5) {$\alpha_{s+1,s}$};
    \node [above, thin] at (8,-5) {$\alpha_{n-1,s}$};
    \node [above, thin] at (9,-5) {$\alpha_{n,s}$};
    
    \foreach \x in {1,...,9}{
    \node [above, thin] at (\x,-3) {$\ldots$};
    }
    \foreach \y in {-5,...,0}{
     \node [above, thin] at (3,\y) {$\ldots$};
     \node [above, thin] at (7,\y) {$\ldots$};
      }
      
       \node[cross,thin, scale=0.8] at (2,-1.5) {};
       \foreach \x in {2,...,8}{
       \node[cross,thin, scale=0.8] at (0.5+\x,-1) {};
       }
     
    \foreach \x in {1,...,2}{
    \foreach \y in {-2,...,0}{
    \fill[fill=black] (\x,\y) circle (0.03 cm);
    }}
    
    \foreach \x in {1,...,2}{
    \foreach \y in {-5,...,-4}{
    \fill[fill=black] (\x,\y) circle (0.03 cm);
    }}
    
    \foreach \x in {4,...,6}{
    \foreach \y in {-2,...,0}{
    \fill[fill=black] (\x,\y) circle (0.03 cm);
    }}
    \foreach \x in {4,...,6}{
    \foreach \y in {-5,...,-4}{
    \fill[fill=black] (\x,\y) circle (0.03 cm);
    }}
    
    \foreach \x in {8,...,9}{
    \foreach \y in {-2,...,0}{
    \fill[fill=black] (\x,\y) circle (0.03 cm);
    }}
    \foreach \x in {8,...,9}{
    \foreach \y in {-5,...,-4}{
    \fill[fill=black] (\x,\y) circle (0.03 cm);
    }}
\foreach \x in {2,...,9}{    
   
    \draw[dashed, -latex] (\x,-1) arc 
    (225:135:cos 45);
     }
     
\foreach \x in {2,...,9}{    
   
  \node [thin] at (\x-0.3,-0.3) {$_{-}$};
    }

\end{tikzpicture}

\end{figure}
We remark that table $T'$ obtained this way is not necessarily in a normal form. However, the only problem that might occur is an almost zero unconstrained column which then gets into a separate table, say, $T'=T''\sqcup T_{x_{i_1}}\ldots\sqcup T_{x_{i_r}}$, where $T''$ is a single table in a normal form and all the other tables are singletons in unconstrained variables $x_{i_1},\ldots, x_{i_r}$. This way we can get a (generalised) table in a normal form for $J'$. Note that passing from $T$ to $T''\sqcup T_{x_{i_1}}\ldots\sqcup T_{x_{i_r}}$ does not change the order of constrained variables: tables $T$ and $T''$ have the same sets and orders of constrained variables modulo the fact of $T$ having an extra $x_1$ in the beginning. Also note that $\dim(S(J))=\dim(S(J'))=\dim(S(J''))=\dim(S(K))-1$.
\end{remark}

\begin{theorem}
\label{thm:uniqueness}
Let $K$ be a table ideal. Then there exists a unique $T$ in a normal form such that $K=K(T)$.
\end{theorem}
\begin{proof}
Recall that we have restricted to table ideals with $S(K)$ connected and $\dim (S(K))=s\ge 1$. 
First of all note that it is enough to show that the order of constrained variables is uniquely determined. After the order is determined, there is a unique way to fill in the entries into the future table. We already know that the first constrained variable is uniquely determined from the ideal. If $s=1$, we are done since there are no other constrained variables. So we will assume $s\ge 2$. Without loss of generality, $x_1$ is the first constrained variable. We can find the values for $d_1$ and $d_1-\alpha_{1,1}$ from the common $x_1$-exponent of all the generators except pure powers. Let $J$ and $J'$ be as in \Cref{rem: induction}. Then from \Cref{rem: induction} we know that $S(J')$ is a union of a connected simplicial complex with a set (possibly empty) of isolated vertices. Having isolated vertices $x_{i_1},\ldots, x_{i_r}$ means that the only minimal generator of $J'$ containing such a variable is the corresponding pure power. Let $J''$ be obtained by removing these lonely variables from the minimal generators of $J'$. Then from \Cref{rem: induction} we know that $J''$ comes from a table in a normal form (equivalently, $S(J'')$ is connected) and satisfies $\dim(S(J''))=\dim(S(K))-1=s-1\ge 1$, which by \Cref{lemma: unique first} implies that there is a unique variable (without loss of generality, $x_2$) which can possibly start a table for $J''$. In particular, $x_2$ starts table $T''$, obtained from some $T$ such that $K=K(T)$. This means, the second variable of any such $T$ is uniquely determined. Then we proceed by induction.
\end{proof}
\section{Generating our data set}\label{sec:data}

How does one represent the data and what is the input form in the machine learning model are important questions to answer when using machine learning algorithms. In this section we discuss our choice of presenting the ideals in data form and how they are generated. We need to take care of generating enough table and enough non-table ideals. 

A monomial ideal $I\subset R$ is represented via the exponent vectors of the monomials. Each monomial in the ideal then gives a vector of length $n$, where $n$ is the number of the variables. We then concatenate the vectors into a single big vector. We call this vector the \emph{ideal vector}.
Concatenation gives different vectors depending on the order the exponent vectors are added, thus we first order the exponent vectors with reverse lexicographic order and then concatenate to have consistency in generating the data. 

The final issue present in this chosen representation is the length of the ideal vectors. The length is entirely dependent on the number of generators of the ideal. One could deal with this kind of data through recursive neural networks for example, but our choices of simple models require the vectors to be of the same length. Therefore we attach zeros in the beginning of the vector to get standard length. We note that when working with 10 or less variables this does not pose an issue, but attempting to have, for example, 27 variables, gives vectors of length 15000. We fixed the length in our test data for each dimension as showed in Table \ref{tab:vlength}.

\begin{table}[H]
    \centering
    \begin{tabular}{||c|c||}
    
      \toprule  Number of variables & Vector length without  the label \\
       \midrule 3& 45\\
       4& 92\\
       5& 180\\
       6& 258\\
       7& 511\\
       8& 624\\
       9& 810\\
       10& 1070\\
       \bottomrule
    \end{tabular}
    \caption{Fixed vector length for ideals between 3 and 10 variables.}
    \label{tab:vlength}
\end{table}

\begin{example}
Let $R=k[x,y,z]$ and let $I=(x^2y,y^3z^2,x^7z)$ be an ideal of $R$. The exponent vectors of the generators of $I$ are $(2,1,0),(0,3,2)$ and $(7,0,1)$. 
\end{example}

The rest of the section describes the families we generate. The code for generating these can be found in \cite{GithubCode}.

\subsection{Generating table ideals}

Given integers $s$ and $n$ with $0\leq s <n$, we generate a (pseudo) random $(n,s)$-table, with maximum free entry bounded by some $N_{max}$. The bound on the free entries is entire due to the requirement of random integer generation requiring bounds. 
From this table we then define an ideal following the definition of generators of table ideals.

\begin{algorithm}[H]
\SetAlgoLined
\KwResult{Exponent vectors of a table ideal generators }
 Use random number generator to generate a table $T$ of size s,n\;
 SET I TO []\;
SET s TO len(T)-1\;
SET n TO len(T[0])\;
SET M TO T.copy()\;
SET mono TO [0 FOR j IN range(n)]\;
 \For{ $i$ IN range($s$)}{
  \If{$i > 0$}{
  SET M[0] TO deduct(M[0], M[i])\;
         SET mono[i-1] TO M[0][i-1]\;
   }
   \For{j IN range(i,n)}{
    SET m TO mono.copy()\;
         SET m[j] TO M[0][j]\;
         I.append(m)\;
   }
 }
 \caption{Generating table ideals}
\end{algorithm}

The main drawback of our code for table ideals is that we cannot guarantee a minimal generating set for the ideal. We use the non-minimal generators for our experiments, but recognise that this could make it easier for the model to learn the pattern than in the case where we would reduce them to minimal generators.

\subsection{Generating non-table ideals}
Table ideals are a restrictive class of monomial ideals, thus most monomial ideals are not table ideals.
It is important to have a wide collection of non-table ideals to ensure that the machine learning algorithms that we train "see" enough variety of non-table ideals and do not get a biased idea of what a non-table ideal is.

We use three families of non-table ideals to get a wide selection of training data. These are table permutations, Macaulay's inverse systems, and a class of ideals from \cite{juhnke2018}. Inverse systems provide us with non-symmetric examples whereas the other two are symmetric.

\section{First approach: Feedforward neural networks}
\label{sec:ML1}

In this section we describe our first ML approach to the classification of table and non-table ideals, which is based on a feedforward neural network.
Artificial neural networks come in many forms, and feedforward networks are one of the simplest and most common ones. A \textit{feedforward neural network} is an artificial neural network (ANN) in which there are no cycles.
It is based on having single input and single output, and the internal layers consist of a linear transformation and a nonlinear component. For a more detailed introduction to ANN see for example \cite{Jung} Section 3.11.

\subsection{Sequential model with ReLU and Sigmoid activation functions}

We first use a neural network with 1 hidden layer and one output layer. The structure of the network is presented in Figure \ref{fig:my_label}. The aim for this network is to learn the label on the ideal from the input, i.e. from the ideal vector. Over the training, the neural network tries to estimate the function from $\mathbb{R}^N$ to $\{0,1\}$, where $N$ is the size of the input vector, as a piece-wise linear function.  
We chose to use the \textit{binary cross entropy} as a loss function, since the task is binary classification, 
and we used Adam \cite{kingma2017adam} as the optimiser, since it is a well performing stochastic gradient descent optimiser that works in most cases.

\tikzset{%
  every neuron/.style={
    circle,
    draw,
    minimum size=1cm
  },
  neuron missing/.style={
    draw=none, 
    scale=4,
    text height=0.333cm,
    execute at begin node=\color{black}$\vdots$
  },
}

\begin{figure}[H]
\begin{center}
\begin{tikzpicture}[x=1.5cm, y=1.5cm, >=stealth]

\foreach \m/\l [count=\y] in {1,2,3,missing,4}
  \node [every neuron/.try, neuron \m/.try] (input-\m) at (0,2.5-\y) {};

\foreach \m [count=\y] in {1,missing,2}
  \node [every neuron/.try, neuron \m/.try ] (hidden-\m) at (2,2-\y*1.25) {};

\foreach \m [count=\y] in {1}
  \node [every neuron/.try, neuron \m/.try ] (output-\m) at (4,0.5-\y) {};

\foreach \l [count=\i] in {1,2,3,N}
  \draw [<-] (input-\i) -- ++(-1,0)
    node [above, midway] {$I_\l$};

\foreach \l [count=\i] in {1,N}
  \node [above] at (hidden-\i.north) {$H_\l$};

\foreach \l [count=\i] in {1}
  \draw [->] (output-\i) -- ++(1,0)
    node [above, midway] {$O_\l$};

\foreach \i in {1,...,4}
  \foreach \j in {1,...,2}
    \draw [->] (input-\i) -- (hidden-\j);

\foreach \i in {1,...,2}
  \foreach \j in {1}
    \draw [->] (hidden-\i) -- (output-\j);

\foreach \l [count=\x from 0] in {Input, Hidden, Ouput}
  \node [align=center, above] at (\x*2,2) {\l \\ layer};

\end{tikzpicture}
\caption{2-layer neural network}
\label{fig:my_label}
\end{center}
\end{figure}
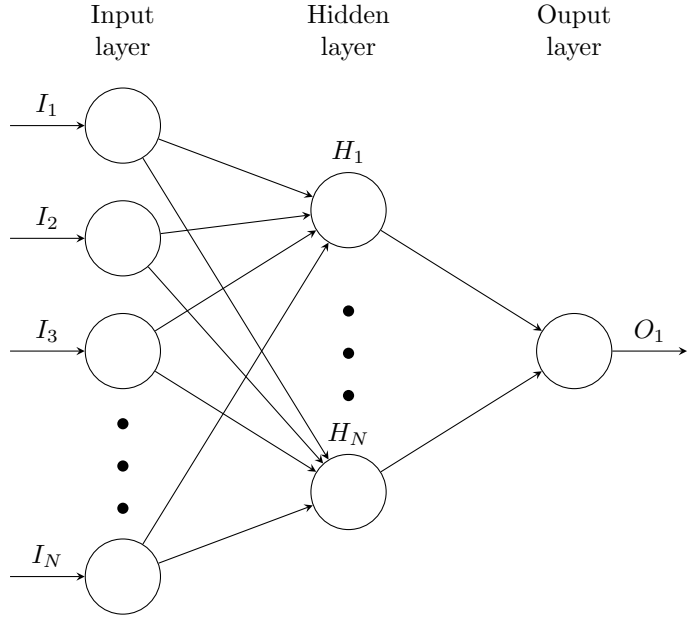

The parameters in this neural network are the entries of the two matrices in the layers, and they are updated on each training round with the gradient descent.

\subsection{Results}
For each number of variables $n\in \{3,\ldots,10\}$, we trained a feed forward neural network.
The data used for this training consisted of all the types in Section \ref{sec:data}, except the almost table ideals. For all cases, the obtained binary accuracy of the model was 1 on both the training set and the test set. The graphs of loss and accuracy from training the $n=5$ model are presented in Figure \ref{fig:train1}. This level accuracy is a bit suspicious, and thus we followed up by testing specific examples that were expected to fail with the trained model.
\begin{figure}[ht]
\begin{center}
    \includegraphics[scale=.5]{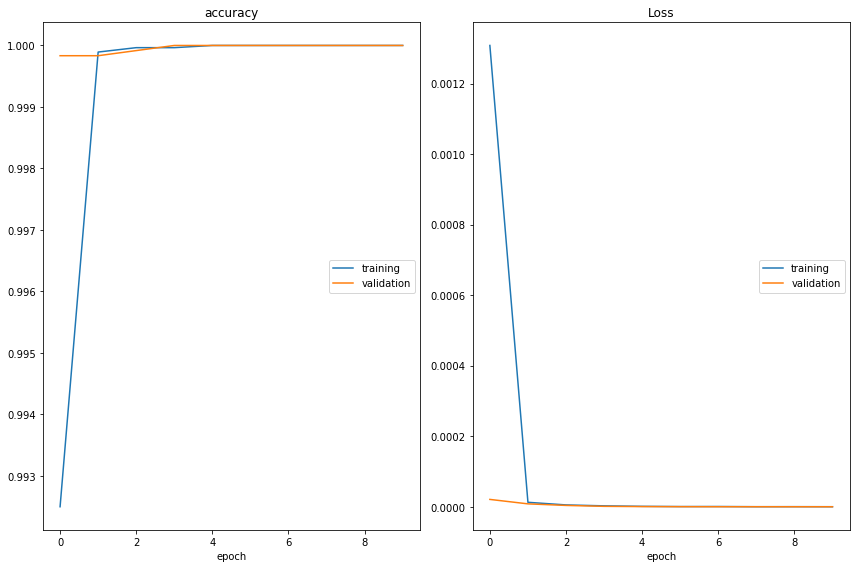}
    \caption{The accuracy and loss graphs from training ANN for 5 variable ideals for 10 epochs. The training values are measured from the training set that the model sees while learning and the validation is from the test set the model does not see during training.}
    \label{fig:train1}
    \end{center}
\end{figure}

Firstly, we added the almost-table ideals. The accuracy of the model dropped to 50\% with all $n$ when trained the same amount of time as the previous models. To try to improve the performance we trained the neural network for a longer time, graphs of this are presented in Figure \ref{fig:train2}. When training for a longer time we see a clear overfitting happening: the training set is being learnt and has a high accuracy but the test set stagnates at the accuracy of 50\%.

\begin{figure}[ht]
    \begin{center}
    \includegraphics[scale=0.5]{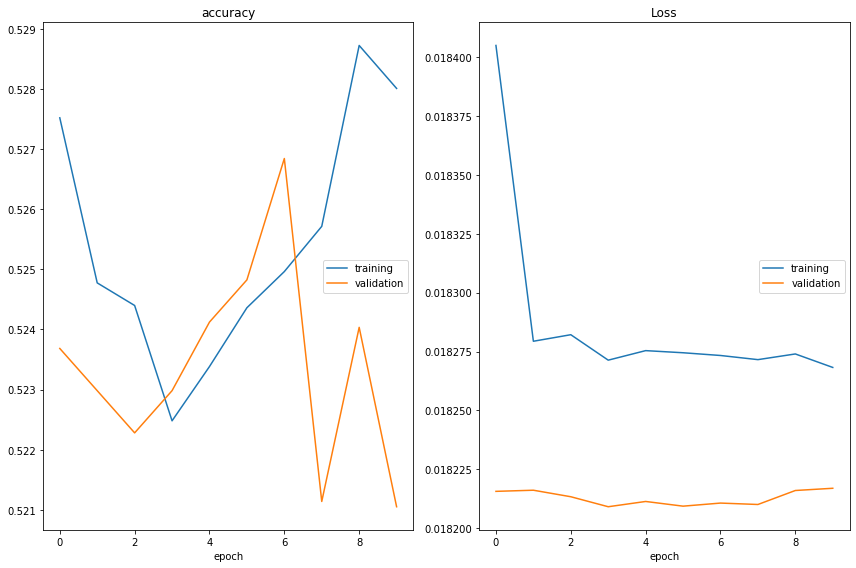}
    \caption{Training 10 var almost table vs table for short time}
    \label{t}
    \end{center}
\end{figure}

\begin{figure}[ht]
    \begin{center}
   \includegraphics[scale=0.5]{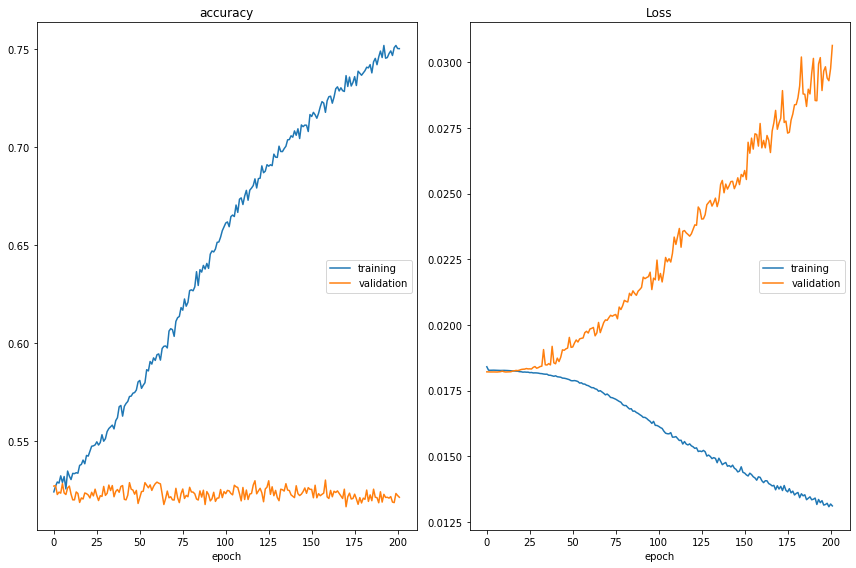}
    \caption{Training 10 variables for 1000 epochs on two layer NN}
    \label{fig:train2}
    \end{center}
\end{figure}

\section{The decision tree}
\label{sec:ML2}
Next we want to apply the decision tree model to our table ideal data. In machine learning, decision trees are non-parametric models that recursively partition the data such that vectors with same labels are grouped together. We use the Scikit-Learn decision tree \cite{scikit-learn}, and the documentation for the decision tree contains also a detailed explanation of the mathematics behind it. 

In simple terms, the partitioning of the data is done by checking each entry in the vector for splitting with respect to some value $t$ and then the one giving  split closest to the actual labels is chosen. The quality of each candidate split is computed using the chosen loss function of the decision tree. 
In our case we used the default loss function that is the Gini loss function. 

The advantage of using a decision tree is that we can see where the data is split and can try to understand the model through these clear splits. Moreover, if the trees are small enough they can be visualised, examples of this are Figure \ref{fig:tree5} and the trees in the Appendix \ref{app:dt_graph}.

Depending on the data, the decision trees can grow to be huge and even include nodes that are not necessary. Thus it is common practice to limit the maximum number of leaves, depth, or other components to keep the tree to a manageable size. In our experiments this was not needed as the trees stayed naturally small.
\begin{figure}[ht]
    \begin{center}
    \includegraphics[scale=0.3]{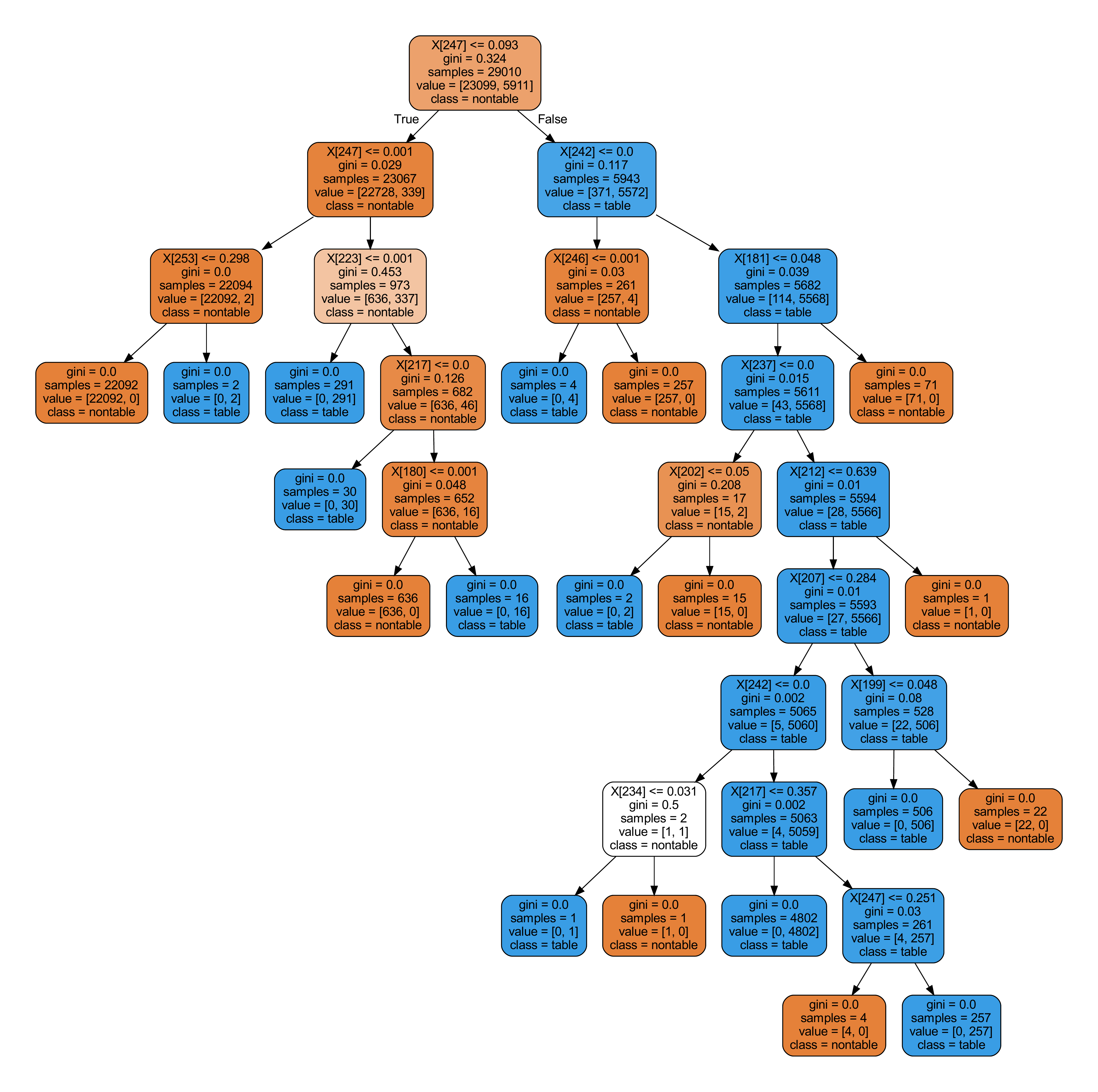}
    \caption{A decision tree of 10 variables}
    \label{fig:tree5}
    \end{center}
\end{figure}

\subsection{Experiments on table ideal data}
We used the same tree model for all number of variables, and let the training process run freely without constraints on the maximum depth, number of leaves, etc.

Decision trees are known to be unstable, so slightly different data gives a different tree. To combat this we computed the average tree over 100 iterations. The average number of nodes, depth, number of leaves and total errors are presented in Table \ref{table:avtree} for 3 to 10 variables.
The average results in Table \ref{table:avtree} show very little variance between the number of variables, which suggest to us that the process of distinguishing table ideals from non-table ideals via the decision tree process is similar to all number of variables. 

 Moreover, looking at the explicit figures of the trees, the first split is fixed and further down the trees often check same points. Since we ordered the monomials in the ideals, the first check corresponds roughly to checking the monomial $x^ay^b$ where $a$ and $b$ are some integers. These observations lead us to conjecture that there exists an algortihmic way to check whether an ideal is a table ideal or not. 
\begin{claim}
There exists an algorithm that can check if a given monomial ideal is a table ideal. 
\end{claim}
This assumption turns out to be true.

\begin{table}
\begin{center}
\begin{tabular}{|ccccc|}
\toprule    n &  number of nodes &  depth &  number of leaves &  total errors \\
\midrule    3 &            45.40 &   9.90 &             23.20 &          7.44 \\
  4 &            31.32 &   7.97 &             16.16 &           6.3 \\
 5 &            44.60 &   8.01 &             22.80 &          8.62 \\
  6 &            48.86 &   9.50 &             24.93 &         10.69 \\
  7 &            55.44 &  12.04 &             28.22 &         14.02 \\
    8 &            48.62 &  11.54 &             24.81 &         12.59 \\
   9 &            40.76 &   9.99 &             20.88 &         11.28 \\
  10 &            36.28 &   8.95 &             18.64 &          9.94 \\
 
 \bottomrule
\end{tabular}
    \caption{Summary of different trees average nodes, depth, number of leaves and total errors}
   \label{table:avtree}
\end{center}
\end{table}

\section{Graph neural networks}
\label{sec:ML3}

Graph neural networks (GNNs) can be thought of as a type of recurrent neural networks where the input data is circulated multiple times through the same subnetwork. The obvious difference to the other models explored in this paper is that GNN's input is a graph. To use GNNs we need to first manually define a graph for the monomial ideals, but the benefit of using graphs is that we do not need to limit the number of generators in the ideals.

For every ideal $I$ we define a graph as follows: let $m_1,m_2,\ldots,m_r$ be the generating set of monomials of the ideal, then the graph $G$ will have $r$ vertices, one corresponding to each monomial, and we label these vertices with the exponent vector of the monomial. We have an edge between the vertices corresponding to $m_i$ and $m_j$ if the support of $m_i$ is contained in the support of $m_j$. See Figure \ref{fig:graph_ex} for an example of such a graph.

\begin{figure}[ht]
    \begin{center}
    \includegraphics[scale=0.5]{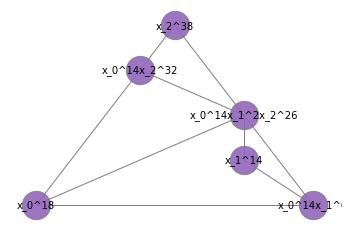}
    \caption{The graph associated to the ideal $I=(x_0^{18},x_1^{14},x_2^{38},x_0^{14}x_2^{32},x_0^{14}x_1^{2}x_2^{26},x_0^{14}x_1^6) $.}
    \label{fig:graph_ex}
    \end{center}
\end{figure}

The graph neural network that we are interested in is the graph classification one. For a more detailed review on graph neural networks see \cite{gnns}. For the code of the GNN we use PyTorch Geometric \cite{pt_geom}. The GNN we use takes as an input the graph described above and has layer of updating the labels, also known as message passing. In the message passing the network looks at the neighbours of a vertex and updates the labels based on the labels of the neighbours. There are many ways to choose the message passing function, i.e. how we update labels based on the information from neighbouring nodes. We used the method known as graph convolutional network \cite{GNN_conv}, where the update function for a label $x_i$ on vertex $v_i$ is 
$$x_i'=\Theta\sum_{j \in N(v_i)\cup \{i\}}\frac{e_{j,i}}{\sqrt{\hat{d_j}\hat{d_i}}}x_j$$
where $\hat{d_i}=1+\sum_{j\in N(v_i)}e_{i,j}$ and $e_i,j$ denotes the edge weight from $v_j$ to $v_i$ that is by default 1, and $\Theta$ denotes the filter matrix that is learnt by the neural network. 
At this stage the GNN has produced a graph with new labels and to do classification we need to go from a labelled graph to a vector. To do this we used the average label from the graph and then had a linear layer to produce the output in the right form to classify the results.

\subsection{Results from GNNs}
We ran the same ideal data experiments on GNN as we did on the previous neural networks. The GNN achieved the same accuracy of 100\% as the others on the crude data set, i.e. where we did not include almost table ideals. On the tests with almost table and table ideals it had an accuracy of 50\%, just as the other NN did. The training graphs are presented in Figures and , respectively.

On the data set that did not contain almost table ideals we see again high accuracy in both training and test sets after just few rounds of training the model. This suggest that graph based approach is efficient in picking up the \say{shape} of the ideal. Looking at the graphs in Figure , we see that introducing the almost table ideals again drops the accuracy and further training does not improve the situation. 

\section{An algorithm to classify table vs non-table ideals}
\label{sec:alg}
Here we present an algorithm which determines whether a given monomial ideal $K$ given by its unique minimal generating set is a table ideal in a suitable ring, and produces a proper generalised table in a normal form in case the answer is positive. First of all, we will only consider Artinian ideals $K$, that is, ideals such that among the minimal generators we have pure powers of all the variables. Clearly, $K$ can not be a table ideal otherwise. Also, without loss of generality we will assume that $S(K)$ is connected, otherwise we treat each connected component separately.
\begin{enumerate}

\item Determine the size of a future table, the set of constrained variables and inductively determine their order. Recall that the order on unconstrained variables does not matter. If at some point determining the next variable is impossible, the ideal is not a table ideal. Without loss of generality, the order of constrained variables is the default order. Note that if the order is different, it is easier to rename variables $x_{i_k}\rightarrow y_k$ and change the whole ideal $K$ than to keep track of the permutation on $x_i$.
\item Fill in the numbers into a matrix row by row. If we fill them in this order, then each $m_{i,j}$ will allow us to compute $\alpha_{1,1}, \ldots, \alpha_{i,i}$ and $\alpha_{i,j}$. If a generator with a particular support is missing, we put a zero into the matrix. 
\item When the matrix is filled, check whether the table conditions are satisfied.
\end{enumerate}


\begin{example}
Let $K=(x_1^4, x_2^5, x_3^6, x_1^2x_2, x_1^3x_3$. Then:
\begin{enumerate}
\item The table, if it existed, would have $1$ colour and $3$ variables. However, $F(K)=\emptyset$ and thus $K$ is not a table ideal. 
\end{enumerate}
\end{example}
\begin{example}
Let $K=(x_1^9, x_2^5, x_3^7, x_1^6x_2^2, x_1^ax_2^2x_3)$, where $0<a<6$. Then:
\begin{enumerate}
\item Our future table has to have $2$ colours and $3$ variables. Clearly, $F(K)=\{x_2\}$. Therefore, $x_2$ is the first constrained variable. Now, $K:(x_2^2)=(x_1^6,x_2^3,x_3^7,x_1^ax_3)$. Removing all the pure powers of lonely variables, we will get $K_1=(x_1^6,x_3^7,x_1^ax_3)$. Now we have $F(K_1)=\{x_1,x_3\}$, but we choose $x_1$ as the next constrained variable since the $x_1$-pure power has exponent $6$ and $x_3$-pure power has exponent $7$ and $6<7$. Note that if these exponents were equal, we would conclude that the ideal is not a table ideal.
\item We now know the order on constrained variables. The ideal under $x_2\rightarrow y_1$, $x_1\rightarrow y_2$, $x_3\rightarrow y_3$ becomes $K'=(y_1^5,y_2^9,y_3^7,y_1^2y_2^6,y_1^2y_2^ay_3)$  and the corresponding matrix will look like this:
\begin{figure}[H]
\centering
\begin{tikzpicture}[scale=2]
 
\draw[red, dashed] (1,0) -- (2,-1) -- (2,-2) -- (2,-1) -- (3,-1);
\draw[green, dashed] (2,0) -- (2,-1) -- (3,-2);
\node [above, thin] at (1,0) {$5$};
\node [above, thin] at (2,0) {$9$};
\node [above, thin] at (3,0) {$7$};
\node [above, thin] at (1,-1) {$3$};
\node [above, thin] at (2,-1) {$3$};
\node [above, thin] at (3,-1) {$0$};
\node [above, thin] at (1,-2) {$0$};
\node [above, thin] at (2,-2) {$6-a$};
\node [above, thin] at (3,-2) {$6$};
\foreach \x in {1,...,3}{
    \foreach \y in {-2,...,0}{
    \fill[fill=black] (\x,\y) circle (0.03 cm);
    }}

\end{tikzpicture}
\caption{}
\end{figure}
This is a table if and only if $a=4$.  

\end{enumerate}
\end{example}

\begin{example}
Let $K=(x_1^2,x_2^4,x_3^7,x_4^8,x_5^{10}, x_1x_2^2x_3^3x_4x_5^2)$. Then:
\begin{enumerate}
\item Our table has to have $4$ colours and $5$ variables. Now, 
$F(K)=\{x_1,x_2,x_3,x_4,x_5\}$, but $x_1$ has the smallest exponent in the corresponding pure power. Then $K:(x_1)=(x_1,x_2^4,x_3^7,x_4^8,x_5^{10}, x_2^2x_3^3x_4x_5^2)$. Removing the isolated variable $x_1$, we will get $K_1=(x_2^4,x_3^7,x_4^8,x_5^{10}, x_2^2x_3^3x_4x_5^2)$. Now, $F(K_1)=\{x_2,x_3,x_4,x_5\}$, but $x_2$ has the smallest exponent in the corresponding pure power. Then $K_1:(x_2^2)=(x_2^2,x_3^7,x_4^8,x_5^{10}, x_3^3x_4x_5^2)$. Removing the isolated variable $x_2$, we get $K_2=(x_3^7,x_4^8,x_5^{10}, x_3^3x_4x_5^2)$. Now,
$F(K_2)=\{x_3,x_4,x_5\}$, but $x_3$ has the smallest exponent in the corresponding pure power. Then $K_2:(x_3^3)=(x_3^4,x_4^8,x_5^{10},x_4x_5^2)$. Removing the isolated variable $x_3$, we get $K_3=(x_4^8,x_5^{10},x_4x_5^2)$. Finally, $F(K_3)=\{x_4,x_5\}$, but $x_4$ has the smallest exponent in the corresponding pure power. We conclude that the order of variables is the default order. We could have concluded it already noting that $F(K)=\{x_1,x_2,x_3,x_4,x_5\}$ and that $2<4<7<8<10$.

\item We will fill in the future table row by row. Since we have no monomials with signatures $(1,j)$, we have $\alpha_{1,2}=\alpha_{1,3}=\alpha_{1,4}=\alpha_{1,5}=0$. We do not yet know what $\alpha_{1,1}$ is. Similarly, $\alpha_{2,3}=\alpha_{2,4}=\alpha_{2,5}=0$ and we do not know yet what $\alpha_{2,2}$ is. Analogously, $\alpha_{3,4}=\alpha_{3,5}=0$ and $\alpha_{3,3}$ is unknown so far. At this point our future table looks like this: 
\begin{figure}[H]
\centering
\begin{tikzpicture}[scale=2]
 
\node [above, thin] at (1,0) {$2$};
\node [above, thin] at (2,0) {$4$};
\node [above, thin] at (3,0) {$7$};
\node [above, thin] at (4,0) {$8$};
\node [above, thin] at (5,0) {$10$};
\node [above, thin] at (1,-1) {$\alpha_{1,1}$};
\node [above, thin] at (2,-1) {$0$};
\node [above, thin] at (3,-1) {$0$};
\node [above, thin] at (4,-1) {$0$};
\node [above, thin] at (5,-1) {$0$};
\node [above, thin] at (1,-2) {$0$};
\node [above, thin] at (2,-2) {$\alpha_{2,2}$};
\node [above, thin] at (3,-2) {$0$};
\node [above, thin] at (4,-2) {$0$};
\node [above, thin] at (5,-2) {$0$};
\node [above, thin] at (1,-3) {$0$};
\node [above, thin] at (2,-3) {$0$};
\node [above, thin] at (3,-3) {$\alpha_{3,3}$};
\node [above, thin] at (4,-3) {$0$};
\node [above, thin] at (5,-3) {$0$};
\node [above, thin] at (1,-4) {$0$};
\node [above, thin] at (2,-4) {$0$};
\node [above, thin] at (3,-4) {$0$};
\node [above, thin] at (4,-4) {$\alpha_{4,4}$};
\node [above, thin] at (5,-4) {$\alpha_{4,5}$};

\foreach \x in {1,...,5}{
    \foreach \y in {-4,...,0}{
    \fill[fill=black] (\x,\y) circle (0.03 cm);
    }}

\end{tikzpicture}
\caption{}
\end{figure}
Finally, $x_1x_2^2x_3^3x_4x_5^2$ gives $\alpha_{1,1}=1$, $\alpha_{2,2}=2$, $\alpha_{3,3}=4$, $\alpha_{4,4}=7$ and $\alpha_{4,5}=8$
\item All the table conditions are satisfied. Therefore, $K=K(T)$ where $T$ the following table:

\begin{figure}[H]
\centering
\begin{tikzpicture}[scale=2]
\draw [red, thick] (1,0) -- (2,-1) -- (2,-2) -- (2,-1) -- (5,-1);
\draw [green, thick] (2,0) -- (2,-1) -- (3,-2) -- (3,-3) -- (3,-2) -- (5,-2);
\draw [cyan, thick] (3,0) -- (3,-2) -- (4,-3) -- (4,-4) -- (4,-3) -- (5,-3);
\draw [yellow, thick] (4,0) -- (4,-3) -- (5,-4);
 
\node [above, thin] at (1,0) {$2$};
\node [above, thin] at (2,0) {$4$};
\node [above, thin] at (3,0) {$7$};
\node [above, thin] at (4,0) {$8$};
\node [above, thin] at (5,0) {$10$};
\node [above, thin] at (1,-1) {$1$};
\node [above, thin] at (2,-1) {$0$};
\node [above, thin] at (3,-1) {$0$};
\node [above, thin] at (4,-1) {$0$};
\node [above, thin] at (5,-1) {$0$};
\node [above, thin] at (1,-2) {$0$};
\node [above, thin] at (2,-2) {$2$};
\node [above, thin] at (3,-2) {$0$};
\node [above, thin] at (4,-2) {$0$};
\node [above, thin] at (5,-2) {$0$};
\node [above, thin] at (1,-3) {$0$};
\node [above, thin] at (2,-3) {$0$};
\node [above, thin] at (3,-3) {$4$};
\node [above, thin] at (4,-3) {$0$};
\node [above, thin] at (5,-3) {$0$};
\node [above, thin] at (1,-4) {$0$};
\node [above, thin] at (2,-4) {$0$};
\node [above, thin] at (3,-4) {$0$};
\node [above, thin] at (4,-4) {$7$};
\node [above, thin] at (5,-4) {$8$};

\foreach \x in {1,...,5}{
    \foreach \y in {-4,...,0}{
    \fill[fill=black] (\x,\y) circle (0.03 cm);
    }}

\end{tikzpicture}
\caption{}
\end{figure}

\end{enumerate}
\end{example}

\bibliography{References_ML-algebra}
\bibliographystyle{alpha, maxbibnames=99}

\newpage
\appendix
\section{Reduction of tables}
\label{reduction}
Given that several tables might give the same ideal, we are aiming for some sort of minimality. For example, for a given table ideal we would like to use a disjoint union of tables with the smallest possible number of total entries. We will see what kinds of irregularities might occur and how tables will be reduced.

\begin{enumerate}

\item Assume $\alpha_{i,i}=0$. We erase the last $s-i+1$ colours and remove the last $s-i+1$ rows of the table. We will call it "cutting at $\alpha_{i,i}=0$". Note that after this procedure columns $i,\ldots, n$ will become unconstrained.

\begin{figure}[H]
\begin{center}
\begin{tikzpicture}[scale=1.4]
      \draw [red,thick]  (2,-1) -- (9,-1);
      \draw [red,thick]  (2,-1) -- (2,-2);
      \draw [red,thick]  (2,-1) -- (1,0);
      
      \draw[green, thick]  (2,-1) -- (3,-2) -- (9,-2);
      \draw[green, thick]  (3,-2) -- (3,-3);
      \draw[green, thick]  (2,-1) -- (2, 0);

      \draw[cyan, thick]  (3,-1) -- (3,-2) -- (4,-3)--(9,-3);
      \draw[cyan, thick]  (4,-4) -- (4,-3);
      \draw[cyan, thick]  (3,-1) -- (3,0);

      \draw[yellow,thick]  (4,-1) -- (4,-3) -- (5,-4) -- (9,-4);
      \draw[yellow,thick]  (5,-4) -- (5,-5);
      \draw[yellow,thick]  (4,-1) -- (4,0);
      
      \draw[blue,thick]  (5,-1) -- (5,-4) -- (6,-5) -- (9,-5);
      \draw[blue,thick]  (5,-1) -- (5,-0);
	
    \foreach \x in {1,...,2}{
    \node [above, thin] at (\x,0) {$d_\x$};
    \node [above, thin] at (\x,-1) {$\alpha_{1,\x}$};
    
    }
    \node [above, thin] at (1,-2) {$0$};
    \node [above, thin] at (2,-2) {$\alpha_{2,2}$};
    

    \node [above, thin] at (4,-1) {$\alpha_{1,s-1}$};
    \node [above, thin] at (5,-1) {$\alpha_{1,s}$};
    \node [above, thin] at (6,-1) {$\alpha_{1,s+1}$};
    \node [above, thin] at (8,-1) {$\alpha_{1,n-1}$};
    \node [above, thin] at (9,-1) {$\alpha_{1,n}$};
    
    \node [above, thin] at (4,-2) {$\alpha_{2,s-1}$};
    \node [above, thin] at (5,-2) {$\alpha_{2,s}$};
    \node [above, thin] at (6,-2) {$\alpha_{2,s+1}$};
    \node [above, thin] at (8,-2) {$\alpha_{2,n-1}$};
    \node [above, thin] at (9,-2) {$\alpha_{2,n}$};

    \node [above, thin] at (4,0) {$d_{s-1}$};    
    \node [above, thin] at (5,0) {$d_{s}$};  
    \node [above, thin] at (6,0) {$d_{s+1}$};  
    \node [above, thin] at (8,0) {$d_{n-1}$};
    \node [above, thin] at (9,0) {$d_{n}$};

    \node [above, thin] at (1,-4) {$0$};
    \node [above, thin] at (2,-4) {$0$};
    
    \node [left, thin] at (4,-3.85) {\textcolor{cyan}{$0$}};
    
    \node [above, thin] at (5,-4) {$\alpha_{s-1,s}$};
    \node [above, thin] at (6,-4) {$\alpha_{s-1,s+1}$};
    \node [above, thin] at (8,-4) {$\alpha_{s-1,n-1}$};
    \node [above, thin] at (9,-4) {$\alpha_{s-1,n}$};
    \node [above, thin] at (1,-5) {$0$};
    \node [above, thin] at (2,-5) {$0$};
    \node [above, thin] at (4,-5) {$0$};
    \node [above, thin] at (5,-5) {$\alpha_{s,s}$};
    \node [above, thin] at (6,-5) {$\alpha_{s,s+1}$};
    \node [above, thin] at (8,-5) {$\alpha_{s,n-1}$};
    \node [above, thin] at (9,-5) {$\alpha_{s,n}$};
    
    \foreach \x in {1,...,9}{
    \node [above, thin] at (\x,-3) {$\ldots$};
    }
    \foreach \y in {-5,...,0}{
     \node [above, thin] at (3,\y) {$\ldots$};
     \node [above, thin] at (7,\y) {$\ldots$};
    }

    \foreach \x in {1,...,2}{
    \foreach \y in {-2,...,0}{
    \fill[fill=black] (\x,\y) circle (0.03 cm);
    }}
    
    \foreach \x in {1,...,2}{
    \foreach \y in {-5,...,-4}{
    \fill[fill=black] (\x,\y) circle (0.03 cm);
    }}
    
    \foreach \x in {4,...,6}{
    \foreach \y in {-2,...,0}{
    \fill[fill=black] (\x,\y) circle (0.03 cm);
    }}
    \foreach \x in {4,...,6}{
    \foreach \y in {-5,...,-4}{
    \fill[fill=black] (\x,\y) circle (0.03 cm);
    }}
    
    \foreach \x in {8,...,9}{
    \foreach \y in {-2,...,0}{
    \fill[fill=black] (\x,\y) circle (0.03 cm);
    }}
    \foreach \x in {8,...,9}{
    \foreach \y in {-5,...,-4}{
    \fill[fill=black] (\x,\y) circle (0.03 cm);
    }}
    \draw [dashed]  (0.5,-3.5)--(9.5,-3.5);
    \node[thin] at (0.5,-3.5){\ding{34}};
 
  \foreach \x in {4,...,5}{
    \foreach \y in {-2,...,0}{
 \node[cross,thin, scale=0.8] at (\x,\y-0.5) {};
    }}
\end{tikzpicture}
\end{center}
\end{figure}

\item An unconstrained zero column. By a zero column we mean a column $j$ for which $\alpha_{1,j}=\cdots=\alpha_{s,j}=0$. Note that $d_j$ does not have to be $0$ (in fact, if $d_j=0$, the table immediately reduces to the empty table). By \cref{permute}, we can without loss of generality assume that the zero column in question is column $n$. We simply remove this column from our table, together with its adjacent edges, and put $d_n$ into a separate table as shown below. Our table splits into a disjoint union of $2$ tables.

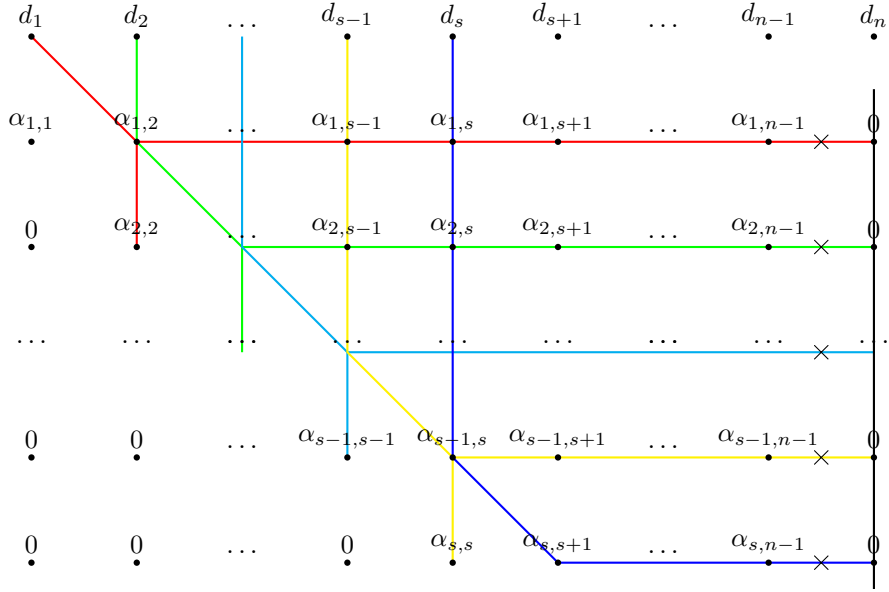
\begin{figure}[ht]

\begin{center}
\begin{tikzpicture}[scale=1.4]
      \draw [red,thick]  (2,-1) -- (9,-1);
      \draw [red,thick]  (2,-1) -- (2,-2);
      \draw [red,thick]  (2,-1)--(1,0);
      
      \draw[green, thick]  (2,-1)-- (3,-2) -- (9,-2);
      \draw[green, thick]  (3,-2) -- (3,-3);
      \draw[green,thick] (2,-1)--(2,0);

      \draw[cyan, thick]  (3,-1)-- (3,-2) -- (4,-3)--(9,-3);
      \draw[cyan, thick]  (4,-4)-- (4,-3);
      \draw[cyan, thick] (3,-1)--(3,0);

      \draw[yellow,thick]  (4,-1)-- (4,-3) --(5,-4)-- (9,-4);
      \draw[yellow,thick]  (5,-4)-- (5,-5);
      \draw[yellow,thick]  (4,-1)-- (4,0);
      \draw[blue,thick]  (5,-1)--(5,-4) --(6,-5)-- (9,-5);
      \draw[blue,thick]  (5,-1)--(5,-0);

    \foreach \x in {1,...,2}{
    \node [above, thin] at (\x,0) {$d_\x$};
    \node [above, thin] at (\x,-1) {$\alpha_{1,\x}$};
    }
    \node [above, thin] at (1,-2) {$0$};
    \node [above, thin] at (2,-2) {$\alpha_{2,2}$};

    \node [above, thin] at (4,-1) {$\alpha_{1,s-1}$};
    \node [above, thin] at (5,-1) {$\alpha_{1,s}$};
    \node [above, thin] at (6,-1) {$\alpha_{1,s+1}$};
    \node [above, thin] at (8,-1) {$\alpha_{1,n-1}$};
    \node [above, thin] at (9,-1) {$0$};
    
    \node [above, thin] at (4,-2) {$\alpha_{2,s-1}$};
    \node [above, thin] at (5,-2) {$\alpha_{2,s}$};
    \node [above, thin] at (6,-2) {$\alpha_{2,s+1}$};
    \node [above, thin] at (8,-2) {$\alpha_{2,n-1}$};
    \node [above, thin] at (9,-2) {$0$};

    \node [above, thin] at (4,0) {$d_{s-1}$};    
    \node [above, thin] at (5,0) {$d_{s}$};  
    \node [above, thin] at (6,0) {$d_{s+1}$};  
    \node [above, thin] at (8,0) {$d_{n-1}$};
    \node [above, thin] at (9,0) {$d_{n}$};

    \node [above, thin] at (1,-4) {$0$};
    \node [above, thin] at (2,-4) {$0$};
    
    \node [above, thin] at (4,-4) {$\alpha_{s-1,s-1}$};
    
    \node [above, thin] at (5,-4) {$\alpha_{s-1,s}$};
    
    \node [above, thin] at (6,-4) {$\alpha_{s-1,s+1}$};
    \node [above, thin] at (8,-4) {$\alpha_{s-1,n-1}$};
    \node [above, thin] at (9,-4) {$0$};

    \node [above, thin] at (1,-5) {$0$};
    \node [above, thin] at (2,-5) {$0$};
    \node [above, thin] at (4,-5) {$0$};
    \node [above, thin] at (5,-5) {$\alpha_{s,s}$};
    \node [above, thin] at (6,-5) {$\alpha_{s,s+1}$};
    \node [above, thin] at (8,-5) {$\alpha_{s,n-1}$};
    \node [above, thin] at (9,-5) {$0$};
    
    \foreach \x in {1,...,9}{
    \node [above, thin] at (\x,-3) {$\ldots$};
    }
    \foreach \y in {-5,...,0}{
     \node [above, thin] at (3,\y) {$\ldots$};
     \node [above, thin] at (7,\y) {$\ldots$};
    }

    \foreach \x in {1,...,2}{
    \foreach \y in {-2,...,0}{
    \fill[fill=black] (\x,\y) circle (0.03 cm);
    }}
    
    \foreach \x in {1,...,2}{
    \foreach \y in {-5,...,-4}{
    \fill[fill=black] (\x,\y) circle (0.03 cm);
    }}
    
    \foreach \x in {4,...,6}{
    \foreach \y in {-2,...,0}{
    \fill[fill=black] (\x,\y) circle (0.03 cm);
    }}
    \foreach \x in {4,...,6}{
    \foreach \y in {-5,...,-4}{
    \fill[fill=black] (\x,\y) circle (0.03 cm);
    }}
    
    \foreach \x in {8,...,9}{
    \foreach \y in {-2,...,0}{
    \fill[fill=black] (\x,\y) circle (0.03 cm);
    }}
    \foreach \x in {8,...,9}{
    \foreach \y in {-5,...,-4}{
    \fill[fill=black] (\x,\y) circle (0.03 cm);
    }}
    
    \foreach \y in {-5,...,-1}{
 \node[cross,thin, scale=0.8] at (8.5,\y) {};
    }
    \draw[thick] (9, -5.25) -- (9, -0.5);
    
\end{tikzpicture}
\caption{Removing a zero column}
\end{center}
\end{figure}

Note that if we have a constrained zero column, say, column $i$, then in particular $\alpha_{i,i}=0$. In this case we can first cut the table at $\alpha_{i,i}=0$, after which column $i$ becomes unconstrained, and deal with the unconstrained zero column.

We will now see how to reduce a table with a singularity. By singularity we mean that $\alpha_{1,j}+\alpha_{2,j}+\cdots+\alpha_{s,j}=d_j$ for some $i$, in other words, $d_j$ equals the sum of the rest of numbers in column $j$. 

\item Singularity in an unconstrained column. We can without loss of generality assume that column $n$ is singular. We do the following:
\begin{itemize}
    \item Add $\alpha_{s,s}$ to $\alpha_{s-1,s}$, unless $s=1$ (that is, unless the table only has $1$ colour). If $s=1$, subtract $\alpha_{1,1}$ from $\alpha_{0,1}=d_1$.
    \item Remove the last row of the table and erase edges of colour $s$.
    
\end{itemize}

\begin{figure}[ht]

\begin{center}
\begin{tikzpicture}[scale=1.4]

      \draw [red,thick]   (2,-1) -- (9,-1);
      \draw [red,thick]   (2,-1) -- (2,-2);
      \draw [red,thick]   (2,-1) -- (1, 0);
      
      \draw [green, thick]  (2,-1) -- (3,-2) -- (9,-2);
      \draw [green, thick]  (3,-2) -- (3,-3);
      \draw [green, thick]  (2,-1) -- (2, 0);

      \draw [cyan, thick]  (3,-1) -- (3,-2) -- (4,-3) -- (9,-3);
      \draw [cyan, thick]  (4,-4) -- (4,-3);
      \draw [cyan, thick]  (3,-1) -- (3, 0);

      \draw[yellow,thick]  (4,-1) -- (4,-3) -- (5,-4) -- (9,-4);
      \draw[yellow,thick]  (5,-4) -- (5,-5);
      \draw[yellow,thick]  (4,-1) -- (4, 0);
      
      \draw[blue,thick]  (5,-1) -- (5,-4) -- (6,-5) -- (9,-5);
      \draw[blue,thick]  (5,-1) -- (5, 0);

    \foreach \x in {1,...,2}{
    \node [above, thin] at (\x,0) {$d_\x$};
    \node [above, thin] at (\x,-1) {$\alpha_{1,\x}$};
    
    }
    \node [above, thin] at (1,-2) {$0$};
    \node [above, thin] at (2,-2) {$\alpha_{2,2}$};

    \node [above, thin] at (4,-1) {$\alpha_{1,s-1}$};
    \node [above, thin] at (5,-1) {$\alpha_{1,s}$};
    \node [above, thin] at (6,-1) {$\alpha_{1,s+1}$};
    \node [above, thin] at (8,-1) {$\alpha_{1,n-1}$};
    \node [above, thin] at (9,-1) {$\alpha_{1,n}$};
    
    \node [above, thin] at (4,-2) {$\alpha_{2,s-1}$};
    \node [above, thin] at (5,-2) {$\alpha_{2,s}$};
    \node [above, thin] at (6,-2) {$\alpha_{2,s+1}$};
    \node [above, thin] at (8,-2) {$\alpha_{2,n-1}$};
    \node [above, thin] at (9,-2) {$\alpha_{2,n}$};

    \node [above, thin] at (4,0) {$d_{s-1}$};    
    \node [above, thin] at (5,0) {$d_{s}$};  
    \node [above, thin] at (6,0) {$d_{s+1}$};  
    \node [above, thin] at (8,0) {$d_{n-1}$};
    \node [above, thin] at (9.5,0) {$d_{n}=\sum_{i=1}^{i=s}{\alpha_{i,n}}$};

    \node [above, thin] at (1,-4) {$0$};
    \node [above, thin] at (2,-4) {$0$};
    
    \node [above, thin] at (4,-4) {$\alpha_{s-1,s-1}$};
    
    \node [above, thin] at (5,-4) {$\alpha_{s-1,s}$};
    
    \node [above, thin] at (6,-4) {$\alpha_{s-1,s+1}$};
    \node [above, thin] at (8,-4) {$\alpha_{s-1,n-1}$};
    \node [above, thin] at (9,-4) {$\alpha_{s-1,n}$};

    \node [above, thin] at (1,-5) {$0$};
    \node [above, thin] at (2,-5) {$0$};
    \node [above, thin] at (4,-5) {$0$};
    \node [above, thin] at (5,-5) {$\alpha_{s,s}$};
    \node [above, thin] at (6,-5) {$\alpha_{s,s+1}$};
    \node [above, thin] at (8,-5) {$\alpha_{s,n-1}$};
    \node [above, thin] at (9,-5) {$\alpha_{s,n}$};
    
    \foreach \x in {1,...,9}{
    \node [above, thin] at (\x,-3) {$\ldots$};
    }
    \foreach \y in {-5,...,0}{
     \node [above, thin] at (3,\y) {$\ldots$};
      \node [above, thin] at (7,\y) {$\ldots$};
      }

    \foreach \x in {1,...,2}{
    \foreach \y in {-2,...,0}{
    \fill[fill=black] (\x,\y) circle (0.03 cm);
    }}
    
    \foreach \x in {1,...,2}{
    \foreach \y in {-5,...,-4}{
    \fill[fill=black] (\x,\y) circle (0.03 cm);
    }}
    
    \foreach \x in {4,...,6}{
    \foreach \y in {-2,...,0}{
    \fill[fill=black] (\x,\y) circle (0.03 cm);
    }}
    \foreach \x in {4,...,6}{
    \foreach \y in {-5,...,-4}{
    \fill[fill=black] (\x,\y) circle (0.03 cm);
    }}
    
    \foreach \x in {8,...,9}{
    \foreach \y in {-2,...,0}{
    \fill[fill=black] (\x,\y) circle (0.03 cm);
    }}
    \foreach \x in {8,...,9}{
    \foreach \y in {-5,...,-4}{
    \fill[fill=black] (\x,\y) circle (0.03 cm);
    }}
    \draw[dashed] (0.5,-4.5)--(9.5,-4.5);
    \node[thin] at (0.5,-4.5){\ding{34}};
    \foreach \y in {-3,...,0}{
 \node[cross,thin, scale=0.8] at (5,\y-0.5) {};}
  \draw[dashed, -latex] (5,-5) arc (225:135:cos 45);
   \node[thin] at (4.7,-4.3) {$_{+}$}; 

\end{tikzpicture}
\caption{Removing a singularity in an unconstrained column for $s>1$}
\end{center}
\end{figure}

\begin{figure}[ht]

\begin{center}
\begin{tikzpicture}[scale=2]
\draw [red,thick]  (2,-1) -- (4,-1);
\draw [red,thick]  (2,-1) -- (1, 0);
\draw [dashed]  (0.5,-0.5)--(4.5,-0.5);
   \node[thin] at (0.5,-0.5){\ding{34}};
   \draw[dashed, -latex] (1,-1) arc (225:135:cos 45);
	
    \node [above, thin] at (1,0) {$d_1$};
    \node [above, thin] at (2,0) {$d_2$};
    \node [above, thin] at (2,-1) {$\alpha_{1,2}$};
    \node [above, thin] at (1,-1) {$\alpha_{1,1}$};

    \node [above, thin] at (3,-1) {$\ldots$};    
    \node [above, thin] at (4,-1) {$\alpha_{1,n}$};
    

    \node [above, thin] at (3,0) {$\ldots$};    
    \node [above, thin] at (4,0) {$d_{n}=\alpha_{1,n}$};    
     
    \foreach \x in {1,...,2}{
    \foreach \y in {-1,...,0}{
    \fill[fill=black] (\x,\y) circle (0.03 cm);
    }}

    \foreach \y in {-1,...,0}{
    \fill[fill=black] (4,\y) circle (0.03 cm);
    }
   \node[thin] at (0.7,-0.3){$_{-}$};
   
\end{tikzpicture}
\caption{Removing a singularity in an unconstrained column for $s=1$}
\end{center}
\end{figure}

Clearly, the procedure of removing a singularity in column $n$ for $s>1$ does not help if $\alpha_{s,n}=0$. In this case we repeat the procedure as long as needed. 
\item  Singularity in a constrained column. If a singularity occurs in a constrained column $i$, we do the following:
\begin{enumerate}
   \item Place column $i$ to the right of column $n$. Draw edges in such a way that they still indicate which numbers sum up to which. 
 
 \begin{figure}[ht]

\begin{center}
\begin{tikzpicture}[scale=1.4]
      \draw [red,thick]     (2,-1) -- (9,-1);
      \draw [red,thick]    (2,-1) -- (2,-2);
      \draw [red,thick]  (2,-1)--(1,0);
      
      \draw[green, thick]  (2,-1)-- (3,-2) -- (9,-2);
      \draw[green, thick]  (3,-2) -- (3,-3);
      \draw[green,thick] (2,-1)--(2,0);

      \draw[cyan, thick]  (3, 0) -- (3,-1)-- (3,-2) -- (4,-3)--(9,-3)--(9,-4);

     \draw[yellow,thick]  (9, 0) -- (9,-3) -- (8,-4) -- (4,-4)--(4,-5);
      \draw[blue,thick]  (4,-1)--(4,-4) --(5,-5)-- (8,-5);
      \draw[blue,thick]  (4,-1)--(4,-0);

    \foreach \x in {1,...,2}{
    \node [above, thin] at (\x,0) {$d_\x$};
    \node [above, thin] at (\x,-1) {$\alpha_{1,\x}$};
    }
    \node [above, thin] at (1,-2) {$0$};
    \node [above, thin] at (2,-2) {$\alpha_{2,2}$};

    \node [above, thin] at (4,-1) {$\alpha_{1,s}$};
    \node [above, thin] at (5,-1) {$\alpha_{1,s+1}$};
    \node [above, thin] at (7,-1) {$\alpha_{1,n-1}$};
    \node [above, thin] at (8,-1) {$\alpha_{1,n}$};
    \node [above, thin] at (9,-1) {$\alpha_{1,s-1}$};
    
    \node [above, thin] at (4,-2) {$\alpha_{2,s}$};
    \node [above, thin] at (5,-2) {$\alpha_{2,s+1}$};
    \node [above, thin] at (7,-2) {$\alpha_{2,n-1}$};
    \node [above, thin] at (8,-2) {$\alpha_{2,n}$};
    \node [above, thin] at (9,-2) {$\alpha_{2,s-1}$};

    \node [above, thin] at (4,0) {$d_{s}$};    
    \node [above, thin] at (5,0) {$d_{s+1}$};  
    \node [above, thin] at (7,0) {$d_{n-1}$};  
    \node [above, thin] at (8,0) {$d_{n}$};
    \node [above, thin] at (9,0) {$d_{s-1}$};

    \node [above, thin] at (1,-4) {$0$};
    \node [above, thin] at (2,-4) {$0$};
    
    \node [above, thin] at (4,-4) {$\alpha_{s-1,s}$};
    
    \node [above, thin] at (5,-4) {$\alpha_{s-1,s+1}$};
    
    \node [above, thin] at (7,-4) {$\alpha_{s-1,n-1}$};
    \node [above, thin] at (8,-4) {$\alpha_{s-1,n}$};
    \node [above, thin] at (9,-4) {$\alpha_{s-1,s-1}$};

    \node [above, thin] at (1,-5) {$0$};
    \node [above, thin] at (2,-5) {$0$};
    \node [above, thin] at (4,-5) {$\alpha_{s,s}$};
    \node [above, thin] at (5,-5) {$\alpha_{s,s+1}$};
    \node [above, thin] at (7,-5) {$\alpha_{s,n-1}$};
    \node [above, thin] at (8,-5) {$\alpha_{s,n}$};
    \node [above, thin] at (9,-5) {$0$};
    
    \foreach \x in {1,...,9}{
    \node [above, thin] at (\x,-3) {$\ldots$};
    }
    \foreach \y in {-5,...,0}{
     \node [above, thin] at (3,\y) {$\ldots$};
     \node [above, thin] at (6,\y) {$\ldots$};
    }

    \foreach \x in {1,...,2}{
    \foreach \y in {-2,...,0}{
    \fill[fill=black] (\x,\y) circle (0.03 cm);
    }}
    
    \foreach \x in {1,...,2}{
    \foreach \y in {-5,...,-4}{
    \fill[fill=black] (\x,\y) circle (0.03 cm);
    }}
    
    \foreach \x in {4,...,5}{
    \foreach \y in {-2,...,0}{
    \fill[fill=black] (\x,\y) circle (0.03 cm);
    }}
    \foreach \x in {4,...,5}{
    \foreach \y in {-5,...,-4}{
    \fill[fill=black] (\x,\y) circle (0.03 cm);
    }}
    
    \foreach \x in {7,...,9}{
    \foreach \y in {-2,...,0}{
    \fill[fill=black] (\x,\y) circle (0.03 cm);
    }}
    \foreach \x in {7,...,9}{
    \foreach \y in {-5,...,-4}{
    \fill[fill=black] (\x,\y) circle (0.03 cm);
    }}

\end{tikzpicture}
\end{center}
\caption{Step 1}

\end{figure}

\item For each $j\in\{i+1,\ldots, n\}$ add $\alpha_{i,j}$ to $\alpha_{i-1,j}$, unless $i=1$ (that is, unless column $1$ is the singular column in question). If $i=1$, subtract $\alpha_{1,j}$ from $\alpha_{0,j}=d_j$ for each $j\in\{2,\ldots, n\}$. Erase all edges of colour $i$.   
    
\begin{figure}[H]

\centering
\begin{tikzpicture}[scale=1.4]

      \draw [red,thick]     (2,-1) -- (9,-1);
      \draw [red,thick]    (2,-1) -- (2,-2);
      \draw [red,thick]  (2,-1)--(1,0);
      
      \draw[green, thick]  (2,-1)-- (3,-2) -- (9,-2);
      \draw[green, thick]  (3,-2) -- (3,-3);
      \draw[green,thick] (2,-1)--(2,0);

      \draw[cyan, thick]  (3, 0) -- (3,-1)-- (3,-2) -- (4,-3)--(9,-3)--(9,-4);

      \draw[blue,thick]  (4,-1)--(4,-4) --(5,-5)-- (8,-5);
      \draw[blue,thick]  (4,-1)--(4,-0);

    \foreach \x in {1,...,2}{
    \node [above, thin] at (\x,0) {$d_\x$};
    \node [above, thin] at (\x,-1) {$\alpha_{1,\x}$};
    }
    \node [above, thin] at (1,-2) {$0$};
    \node [above, thin] at (2,-2) {$\alpha_{2,2}$};

    \node [above, thin] at (4,-1) {$\alpha_{1,s}$};
    \node [above, thin] at (5,-1) {$\alpha_{1,s+1}$};
    \node [above, thin] at (7,-1) {$\alpha_{1,n-1}$};
    \node [above, thin] at (8,-1) {$\alpha_{1,n}$};
    \node [above, thin] at (9,-1) {$\alpha_{1,s-1}$};
    
    \node [above, thin] at (4,-2) {$\alpha_{2,s}$};
    \node [above, thin] at (5,-2) {$\alpha_{2,s+1}$};
    \node [above, thin] at (7,-2) {$\alpha_{2,n-1}$};
    \node [above, thin] at (8,-2) {$\alpha_{2,n}$};
    \node [above, thin] at (9,-2) {$\alpha_{2,s-1}$};

    \node [above, thin] at (4,0) {$d_{s}$};    
    \node [above, thin] at (5,0) {$d_{s+1}$};  
    \node [above, thin] at (7,0) {$d_{n-1}$};  
    \node [above, thin] at (8,0) {$d_{n}$};
    \node [above, thin] at (9,0) {$d_{s-1}$};

    \node [above, thin] at (1,-4) {$0$};
    \node [above, thin] at (2,-4) {$0$};
    
    \node [above, thin] at (4,-4) {$\alpha_{s-1,s}$};
    
    \node [above, thin] at (5,-4) {$\alpha_{s-1,s+1}$};
    
    \node [above, thin] at (7,-4) {$\alpha_{s-1,n-1}$};
    \node [above, thin] at (8,-4) {$\alpha_{s-1,n}$};
    \node [above, thin] at (9,-4) {$\alpha_{s-1,s-1}$};

    \node [above, thin] at (1,-5) {$0$};
    \node [above, thin] at (2,-5) {$0$};
    \node [above, thin] at (4,-5) {$\alpha_{s,s}$};
    \node [above, thin] at (5,-5) {$\alpha_{s,s+1}$};
    \node [above, thin] at (7,-5) {$\alpha_{s,n-1}$};
    \node [above, thin] at (8,-5) {$\alpha_{s,n}$};
    \node [above, thin] at (9,-5) {$0$};
    
    \foreach \x in {1,...,9}{
    \node [above, thin] at (\x,-3) {$\ldots$};
    }
    \foreach \y in {-5,...,0}{
     \node [above, thin] at (3,\y) {$\ldots$};
     \node [above, thin] at (6,\y) {$\ldots$};
    }

    \foreach \x in {1,...,2}{
    \foreach \y in {-2,...,0}{
    \fill[fill=black] (\x,\y) circle (0.03 cm);
    }}
    
    \foreach \x in {1,...,2}{
    \foreach \y in {-5,...,-4}{
    \fill[fill=black] (\x,\y) circle (0.03 cm);
    }}
    
    \foreach \x in {4,...,5}{
    \foreach \y in {-2,...,0}{
    \fill[fill=black] (\x,\y) circle (0.03 cm);
    }}
    \foreach \x in {4,...,5}{
    \foreach \y in {-5,...,-4}{
    \fill[fill=black] (\x,\y) circle (0.03 cm);
    }}
    
    \foreach \x in {7,...,9}{
    \foreach \y in {-2,...,0}{
    \fill[fill=black] (\x,\y) circle (0.03 cm);
    }}
    \foreach \x in {7,...,9}{
    \foreach \y in {-5,...,-4}{
    \fill[fill=black] (\x,\y) circle (0.03 cm);
    }}
    \foreach \x in {4,...,8}{    
   
    \draw[dashed, -latex] (\x,-4) arc 
    (225:135:cos 45);
     \node [thin] at (\x-0.3,-3.3) {$_{+}$};
     }

\end{tikzpicture}
\caption{Step 2 for $i>1$. For $i=1$, the numbers should be subtracted, not added}

\end{figure}
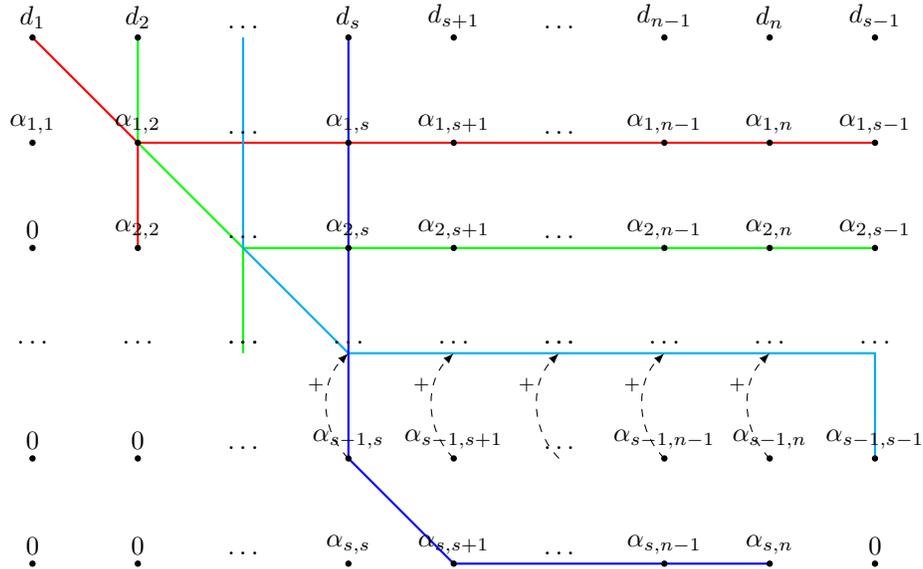    
    
  \item Remove row $i$ and its adjacent edges (the row of $d_j$ is considered row $0$).  
 
\begin{figure}[H]
\begin{center}
\begin{tikzpicture}[scale=1.4]

      \draw [red,thick]     (2,-1) -- (9,-1);
      \draw [red,thick]    (2,-1) -- (2,-2);
      \draw [red,thick]  (2,-1)--(1,0);
      
      \draw[green, thick]  (2,-1)-- (3,-2) -- (9,-2);
      \draw[green, thick]  (3,-2) -- (3,-3);
      \draw[green,thick] (2,-1)--(2,0);

      \draw[cyan, thick]  (3, 0) -- (3,-1)-- (3,-2) -- (4,-3)--(9,-3);

      \draw[blue,thick] (4,0) -- (4,-3); 
      \draw[blue,thick] (5,-4) -- (8,-4);

    \foreach \x in {1,...,2}{
    \node [above, thin] at (\x,0) {$d_\x$};
    \node [above, thin] at (\x,-1) {$\alpha_{1,\x}$};
    }
    \node [above, thin] at (1,-2) {$0$};
    \node [above, thin] at (2,-2) {$\alpha_{2,2}$};

    \node [above, thin] at (4,-1) {$\alpha_{1,s}$};
    \node [above, thin] at (5,-1) {$\alpha_{1,s+1}$};
    \node [above, thin] at (7,-1) {$\alpha_{1,n-1}$};
    \node [above, thin] at (8,-1) {$\alpha_{1,n}$};
    \node [above, thin] at (9,-1) {$\alpha_{1,s-1}$};
    
    \node [above, thin] at (4,-2) {$\alpha_{2,s}$};
    \node [above, thin] at (5,-2) {$\alpha_{2,s+1}$};
    \node [above, thin] at (7,-2) {$\alpha_{2,n-1}$};
    \node [above, thin] at (8,-2) {$\alpha_{2,n}$};
    \node [above, thin] at (9,-2) {$\alpha_{2,s-1}$};

    \node [above, thin] at (4,0) {$d_{s}$};    
    \node [above, thin] at (5,0) {$d_{s+1}$};  
    \node [above, thin] at (7,0) {$d_{n-1}$};  
    \node [above, thin] at (8,0) {$d_{n}$};
    \node [above, thin] at (9,0) {$d_{s-1}$};

    
    
    

    \node [above, thin] at (1,-4) {$0$};
    \node [above, thin] at (2,-4) {$0$};
    \node [above, thin] at (4,-4) {$\alpha_{s,s}$};
    \node [above, thin] at (5,-4) {$\alpha_{s,s+1}$};
    \node [above, thin] at (7,-4) {$\alpha_{s,n-1}$};
    \node [above, thin] at (8,-4) {$\alpha_{s,n}$};
    \node [above, thin] at (9,-4) {$0$};
    
    \foreach \x in {1,...,9}{
    \node [above, thin] at (\x,-3) {$\ldots$};
    }
    \foreach \y in {-4,...,0}{
     \node [above, thin] at (3,\y) {$\ldots$};
     \node [above, thin] at (6,\y) {$\ldots$};
    }

    \foreach \x in {1,...,2}{
    \foreach \y in {-2,...,0}{
    \fill[fill=black] (\x,\y) circle (0.03 cm);
    }}
    
    \foreach \x in {1,...,2}{
    
    \fill[fill=black] (\x,-4) circle (0.03 cm);
    }
    
    \foreach \x in {4,...,5}{
    \foreach \y in {-2,...,0}{
    \fill[fill=black] (\x,\y) circle (0.03 cm);
    }}
    \foreach \x in {4,...,5}{
    
    \fill[fill=black] (\x,-4) circle (0.03 cm);
    }
    
    \foreach \x in {7,...,9}{
    \foreach \y in {-2,...,0}{
    \fill[fill=black] (\x,\y) circle (0.03 cm);
    }}
    \foreach \x in {7,...,9}{
   
    \fill[fill=black] (\x,-4) circle (0.03 cm);
    }

\end{tikzpicture}
\end{center}
\caption{Step 3}

\end{figure}
\item Complete the picture.

\begin{figure}[H]

\begin{center}
\begin{tikzpicture}[scale=1.4]

      \draw [red,thick]     (2,-1) -- (9,-1);
      \draw [red,thick]    (2,-1) -- (2,-2);
      \draw [red,thick]  (2,-1)--(1,0);
      
      \draw[green, thick]  (2,-1)-- (3,-2) -- (9,-2);
      \draw[green, thick]  (3,-2) -- (3,-3);
      \draw[green,thick] (2,-1)--(2,0);

      \draw[cyan, thick]  (3, 0) -- (3,-1)-- (3,-2) -- (4,-3)--(9,-3);
      \draw[cyan, dashed] (4,-3) -- (4,-4);

      \draw[blue,thick] (4,0) -- (4,-3); 
      \draw[blue,thick] (5,-4) -- (8,-4); 
      \draw[blue,dashed] (5,-4) -- (4,-3); 
      \draw[blue,dashed] (9,-4) -- (8,-4);

    \foreach \x in {1,...,2}{
    \node [above, thin] at (\x,0) {$d_\x$};
    \node [above, thin] at (\x,-1) {$\alpha_{1,\x}$};
    }
    \node [above, thin] at (1,-2) {$0$};
    \node [above, thin] at (2,-2) {$\alpha_{2,2}$};

    \node [above, thin] at (4,-1) {$\alpha_{1,s}$};
    \node [above, thin] at (5,-1) {$\alpha_{1,s+1}$};
    \node [above, thin] at (7,-1) {$\alpha_{1,n-1}$};
    \node [above, thin] at (8,-1) {$\alpha_{1,n}$};
    \node [above, thin] at (9,-1) {$\alpha_{1,s-1}$};
    
    \node [above, thin] at (4,-2) {$\alpha_{2,s}$};
    \node [above, thin] at (5,-2) {$\alpha_{2,s+1}$};
    \node [above, thin] at (7,-2) {$\alpha_{2,n-1}$};
    \node [above, thin] at (8,-2) {$\alpha_{2,n}$};
    \node [above, thin] at (9,-2) {$\alpha_{2,s-1}$};

    \node [above, thin] at (4,0) {$d_{s}$};    
    \node [above, thin] at (5,0) {$d_{s+1}$};  
    \node [above, thin] at (7,0) {$d_{n-1}$};  
    \node [above, thin] at (8,0) {$d_{n}$};
    \node [above, thin] at (9,0) {$d_{s-1}$};

    
    
    

    \node [above, thin] at (1,-4) {$0$};
    \node [above, thin] at (2,-4) {$0$};
    \node [above, thin] at (4,-4) {$\alpha_{s,s}$};
    \node [above, thin] at (5,-4) {$\alpha_{s,s+1}$};
    \node [above, thin] at (7,-4) {$\alpha_{s,n-1}$};
    \node [above, thin] at (8,-4) {$\alpha_{s,n}$};
    \node [above, thin] at (9,-4) {$0$};
    
    \foreach \x in {1,...,9}{
    \node [above, thin] at (\x,-3) {$\ldots$};
    }
    \foreach \y in {-4,...,0}{
     \node [above, thin] at (3,\y) {$\ldots$};
     \node [above, thin] at (6,\y) {$\ldots$};
    }

    \foreach \x in {1,...,2}{
    \foreach \y in {-2,...,0}{
    \fill[fill=black] (\x,\y) circle (0.03 cm);
    }}
    
    \foreach \x in {1,...,2}{
    
    \fill[fill=black] (\x,-4) circle (0.03 cm);
    }
    
    \foreach \x in {4,...,5}{
    \foreach \y in {-2,...,0}{
    \fill[fill=black] (\x,\y) circle (0.03 cm);
    }}
    \foreach \x in {4,...,5}{
    
    \fill[fill=black] (\x,-4) circle (0.03 cm);
    }
    
    \foreach \x in {7,...,9}{
    \foreach \y in {-2,...,0}{
    \fill[fill=black] (\x,\y) circle (0.03 cm);
    }}
    \foreach \x in {7,...,9}{
   
    \fill[fill=black] (\x,-4) circle (0.03 cm);
    }

\end{tikzpicture}
\end{center}
\caption{Step 4}

\end{figure}     
    
\end{enumerate}
Note that removing a singularity in column $1$ will result into an unconstrained zero column. We can then reduce the table further according to rule~3 and get a disjoint union of tables.

\end{enumerate}

\section{Decision tree graphs}

\label{app:dt_graph}
\begin{figure}
    \rotatebox{90}{
    \includegraphics[scale=0.3]{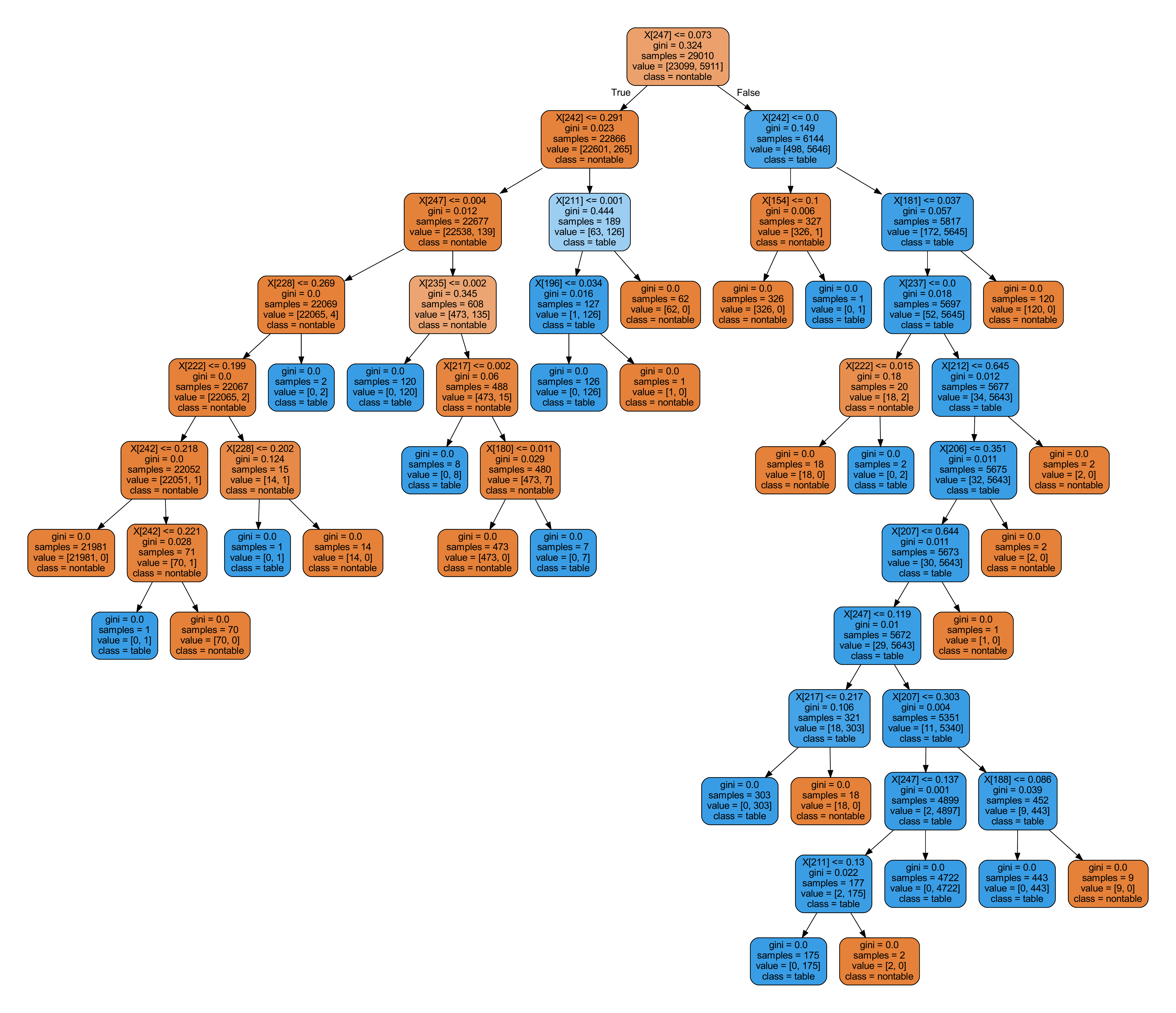}}
    \caption{Decision tree for 10 variables}
    \label{fig:tree2}
\end{figure}

\begin{figure}
    \rotatebox{90}{
    
    \includegraphics[scale=0.4]{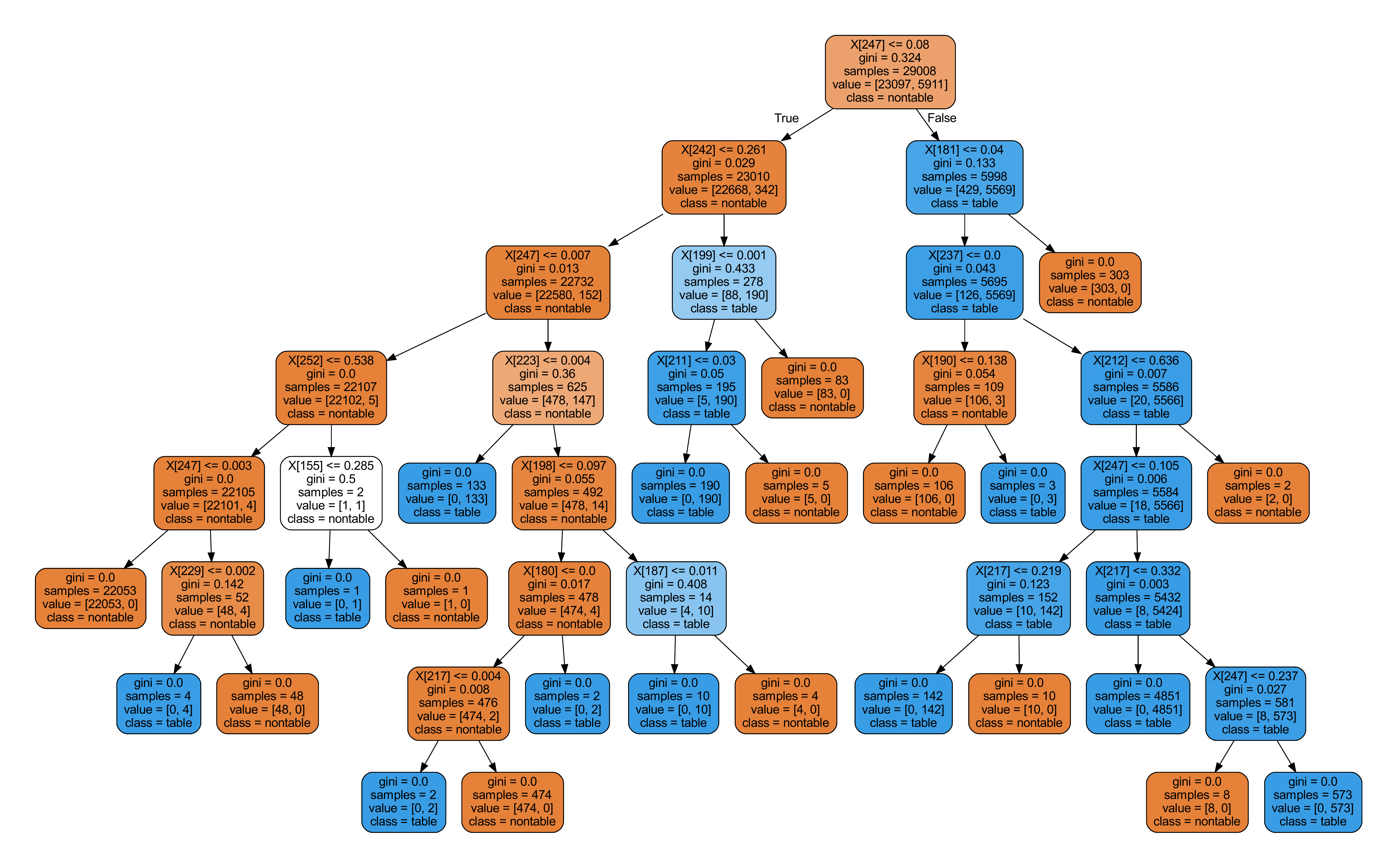}}
    \caption{Decision tree for 10 variables}
    \label{fig:tree3}
\end{figure}

\begin{figure}
    \rotatebox{90}{
    \includegraphics[scale=0.3]{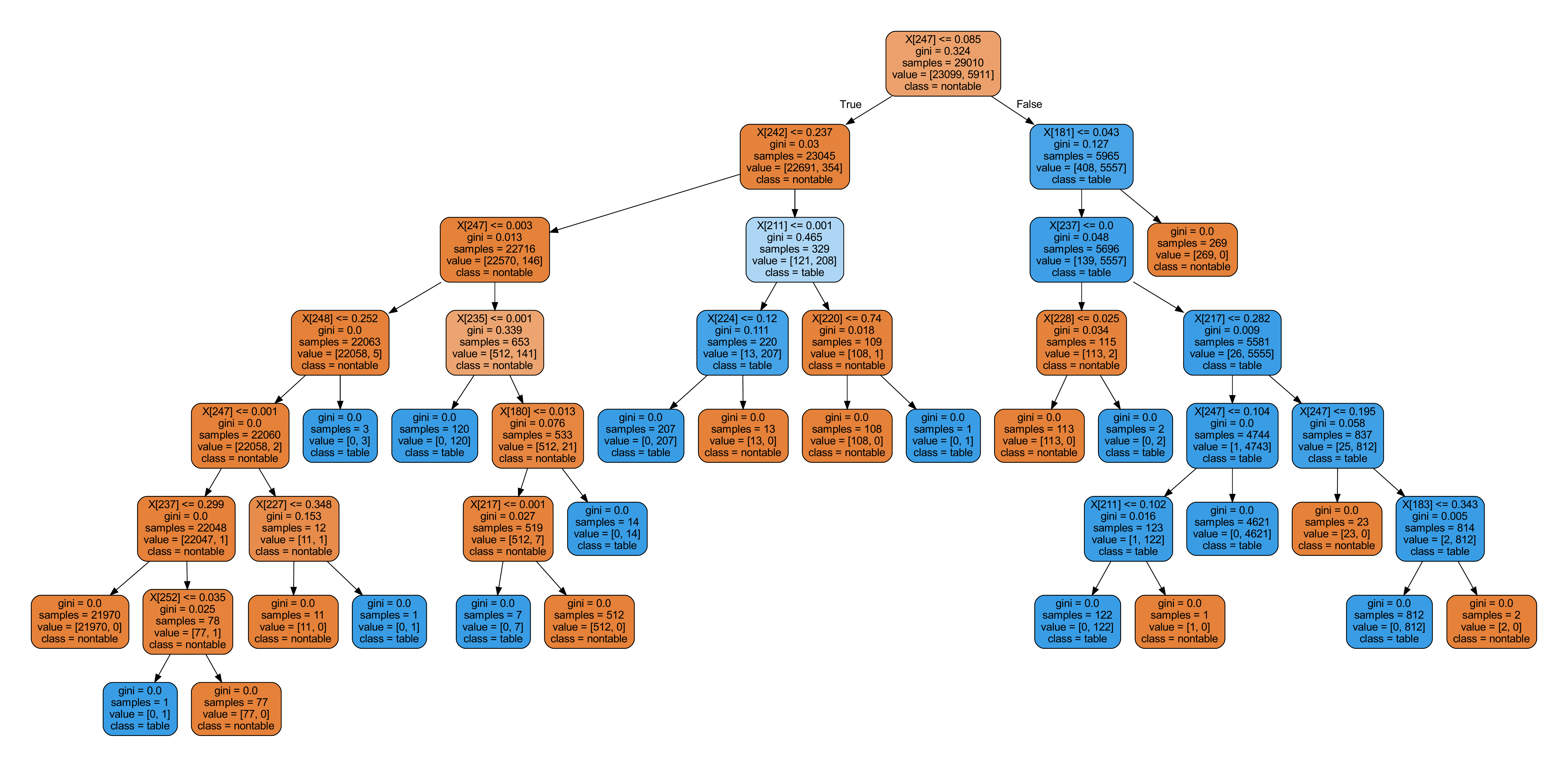}}
    \caption{Decision tree for 10 variables}
    \label{fig:tree4}
\end{figure}

\begin{figure}
    \includegraphics[scale=0.3]{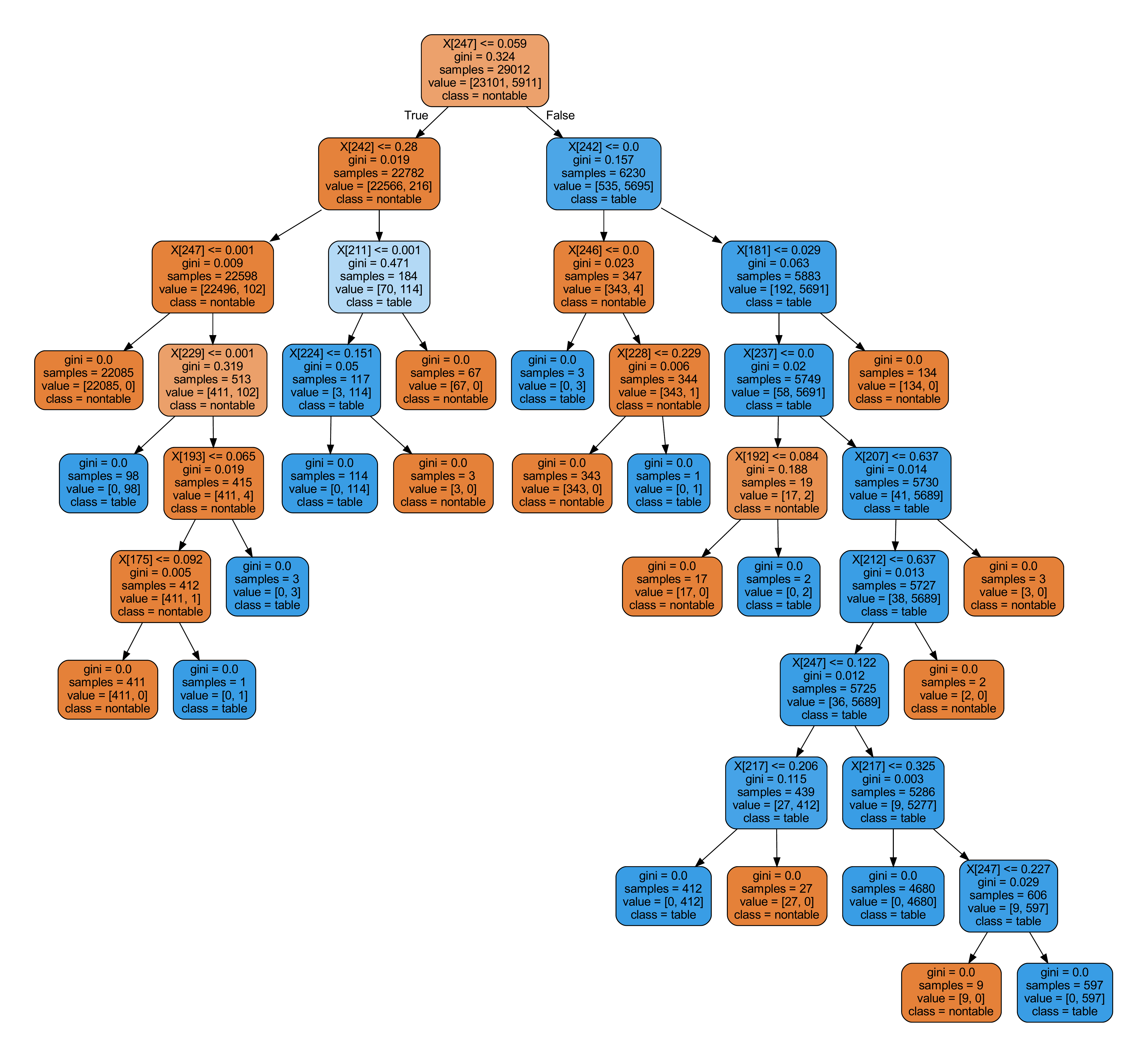}
    \caption{Decision tree for 10 variables}
    \label{fig:tree6}
\end{figure}

\end{document}